\documentclass[10pt,draftcls,onecolumn]{IEEEtran}
\usepackage{setspace}
\usepackage[utf8]{inputenc}
\usepackage{mathtools}
\mathtoolsset{showonlyrefs,showmanualtags}
\usepackage{amssymb,amsthm}
\usepackage{dsfont}
\usepackage{booktabs}
\usepackage{color}
\usepackage{bbold} 
\usepackage[yyyymmdd,hhmmss]{datetime}
\usepackage{bm,upgreek}
\usepackage{tablefootnote}
\usepackage{nth}

 \usepackage{tikz}
\usetikzlibrary{graphs,graphs.standard}
\usetikzlibrary{shapes,snakes}

\DeclareMathOperator{\VAR}{var}
\DeclareMathOperator{\Proj}{proj}
\DeclareMathOperator{\TR}{Tr}
\DeclareMathOperator{\DIAG}{diag}

\DeclareMathOperator{\ROW}{row}

\DeclareMathOperator{\VEC}{vec}
\newcommand{\Norm}[2]{ \| #1 \|_{#2} }
\newcommand{\Exp}[1]{\mathbb{E}[ #1]}
\newcommand{\Zero}{ \mathbf 0 }   
 
\newcommand{\ID}[1]{ \mathbb{1} (#1 ) }

\newtheorem{Assumption}{Assumption}
\newtheorem{Theorem}{Theorem}
\newtheorem{Definition}{Definition}
\newtheorem{Remark}{Remark}

\newtheorem{Lemma}{Lemma}
\newtheorem{Corollary}{Corollary}
\newtheorem{Proposition}{Proposition}
\newtheorem{Problem}{Problem}

\newcommand{\Compress}{\medmuskip=0mu
\thinmuskip=0mu
\thickmuskip=0mu}

\tikzset{%
  every neuron/.style={
    circle,
    draw,
    minimum size=.7cm
  },
  neuron missing/.style={
    draw=none, 
    scale=3,
    text height=0.333cm,
    execute at begin node=\color{black}$\vdots$
  },
}

\definecolor{red}{RGB}{187,0,0}
\definecolor{blue}{RGB}{0, 0,180}
\definecolor{pink}{RGB}{203, 76, 178}
\begin{document}

%

%

\title{\huge Deep Structured Teams with Linear Quadratic
Model: Partial Equivariance and Gauge
Transformation}


\author{Jalal Arabneydi and Amir G. Aghdam
\thanks{This work has been supported in part by the Natural Sciences and Engineering Research Council of Canada (NSERC) under Grant RGPIN-262127-17, and in part by Concordia University under Horizon Postdoctoral Fellowship.}  
\thanks{Jalal Arabneydi and Amir G. Aghdam are with the  Department of Electrical and Computer Engineering, 
        Concordia University, 1455 de Maisonneuve Blvd. West, Montreal, QC, Canada, Postal Code: H3G 1M8.
        {\tt\small Email:jalal.arabneydi@mail.mcgill.ca} and
       {\tt\small Email:aghdam@ece.concordia.ca}}%
}

\maketitle

\begin{abstract}
Motivated by the recent developments in artificial intelligence, we introduce  linear quadratic  deep structured teams in this paper.  Two notions of equivariant and partially equivariant systems  are defined, and it is shown that such systems can be   partitioned into  a few sub-populations of decision makers, where every decision maker in each sub-population is coupled in both dynamics and cost function through a set of  linear  regressions of  the states and actions of all decision makers.  Two non-classical information structures are considered: deep-state sharing and partial deep-state sharing, where deep state  refers to the linear regression of the states  of the decision makers in each sub-population.  For a  risk-sensitive cost function with deep-state sharing structure,  a  closed-form  low-complexity representation  of the  globally optimal strategy is obtained,  whose computational complexity   is independent of the number of decision makers in each sub-population. In addition,  it is shown that the risk-sensitive solution converges to the risk-neutral one as the number of decision makers increases to infinity.  Moreover,  two sub-optimal  sequential strategies under partial  deep-state sharing  information structure are proposed  by introducing  two  Kalman-like filters,  one based on the finite-population model and the other one based on the infinite-population model.  It is  proved that the  prices  of  information  associated with  the above  sub-optimal solutions   converge to zero as the number of decision makers goes to infinity. Furthermore,  a class of feed-forward deep neural  networks   with multiple  layers of weighted sums and products is  introduced   wherein  the  optimal weights and biases  are explicitly obtained.  A supply-chain management example is  presented to  demonstrate the efficacy of  the  obtained results. 
\end{abstract}

\begin{IEEEkeywords}
Networked systems,  large-scale systems,  stochastic control, risk-sensitive cost function,  control applications.
\end{IEEEkeywords}

\section{Introduction}
Recently, there has been a surge of interest  in complex  networked systems such as Internet of Things, social networks, and smart grids.   Such systems often consist of  a large number of interconnected decision makers wherein  a single decision maker (microscopic entity)  has the potential to alter  the  behaviour of the entire (macroscopic) system; a phenomenon known as butterfly effect.  To avoid chaotic situations in such applications,  a social welfare cost function  is often defined as a common cost  for all the decision makers in order to  enforce  the  desired behavior of the  system.   In real-world applications,  it is  also important  to take into account  practical  constraints such as  the privacy of  decision makers,  limited  computational power of processors, and  limited capacity of communications. Under these constraints, the minimization of the above-mentioned cost function becomes   a huge non-convex non-Bayesian optimization problem, whose exploration space grows exponentially  with the control  horizon and the number of  decision makers (which is intractable). As a result,  it is of vital  importance to devise  a systematic approach  to modify  the classical results  in such a way that  they can be efficiently applied to modern control systems consisting of  many interconnected decision makers.

 Inspired by  new architectural developments in  deep neural networks that  have shown a remarkable performance in analyzing big data~\cite{lecun2015deep,schmidhuber2015deep},   we strive to provide an analogous framework in optimal control theory  in order to solve problems with a large number of decision makers.  Recently, the authors have introduced \emph{deep structured teams}  in~\cite{Jalal2019MFT} to study a class of large  Markov decision processes  with discrete  state and action spaces. In this paper, we  propose a  deep structured team  with continuous state and action spaces.  We call such systems deep structured because  the decision makers interact  with each other through  a number of linear regressions (weighted averages) of the states and actions, which resembles  the interactions between the neurons of a deep feed-forward neural network.

  In particular, we consider a class of decentralized control systems that consist of  multiple sub-populations of decision makers, where  each decision maker  is coupled to other  decision makers  in dynamics and cost function  through a few  linear regressions of the states and actions of all decision makers.  Such systems often  arise in practical applications wherein  the  interaction  between the decision makers is  modelled by a set of linear regressions, serving as an approximation for  some highly   complex couplings.  Similar to~\cite{Jalal2019MFT}, two non-classical information structures are considered: deep-state sharing and partial deep-state sharing, where deep state refers to the linear regression of states. In the former information structure, every decision maker observes its local state and  the  joint deep state while in the latter structure, the deep states of  a subset (possibly all) of sub-populations are not observed.  By using  a gauge transformation and   proposing a carefully constructed ansatz, we  find a low-dimensional solution for the  Hamilton-Jacobi-Bellman (HJB) equation in terms of a   Riccati equation (that consists of several local  Riccati equations and one global Riccati equation).  The salient feature of the  Riccati equation is that  its dimension is independent of the number of decision makers in each sub-population,  and hence, the optimal solution is scalable. It is to be noted that the solution itself  depends on  the number of decision makers. In addition, we propose two Kalman-like filters to compute two sub-optimal solutions, one based on the finite-population model and one based on the infinite-population one.  Furthermore,   we show that the prices of   robustness and information converge to zero, as the number of decision makers tends to infinity.  The main results  of this work are also  extended to the  infinite-horizon case.
  

The remainder of the paper is organized as follows. In Section~\ref{sec:invariance}, two notions of equivariant and partial equivariant systems are defined, and  in Section~\ref{sec:problem},   the problem of linear quadratic deep structured  teams  is formulated.   In Section~\ref{sec:main},   the optimal solution   under deep-state sharing information structure   is  obtained, and an explicit connection with a deep  feed-forward  neural network is presented. In Section~\ref{sec:main_pdss},   two sub-optimal solutions under partial deep-state sharing  structure  are  proposed.  In Section~\ref{sec:generalizations}, the main results  are  generalized to  the  infinite-horizon cost function. A supply-chain management  example is provided   in Section~\ref{sec:numerical}, and   some concluding remarks are given   in Section~\ref{sec:conclusion}.

\subsection*{Notation}
In this paper,  the sets of real  and natural numbers are, respectively, denoted by $\mathbb{R}$  and $\mathbb{N}$.  For any vectors $x$, $y$ and $z$,   $\VEC(x, y, z)$  denotes the vector $[x^\intercal, y^\intercal, z^\intercal]^\intercal$, and for any matrices $A,B$ and $C$ with the same number of columns, $\ROW(A,B,C)$ denotes the matrix $[A^\intercal,B^\intercal,C^\intercal ]^\intercal$.  For  any $n \in \mathbb{N}$, $\mathbb{N}_n$ is  the set $\{1,\ldots,n\}$.  Given any square matrices $\Lambda_1$, $\Lambda_2$ and $\Lambda_3$, $\DIAG(\Lambda_1, \Lambda_2, \Lambda_3)$  denotes the block diagonal  matrix with matrices $\Lambda_1$, $\Lambda_2$ and $\Lambda_3$. Moreover, $\DIAG(A)_n$ is a block diagonal matrix with $n \in \mathbb{N}$ copies of the square matrix $A$ on its main diagonal.  In addition, $\TR(\boldsymbol \cdot)$ is the trace of a matrix, $\Exp{\boldsymbol \cdot}$  is the expectation of an event,  $\VAR(\boldsymbol \cdot)$ is the variance of a random variable, and $\ID{\boldsymbol \cdot}$ is the indicator function of a set.  Furthermore,  identity matrix is denoted by $\mathbf{I}$, and  a matrix with all  zero (one) entries is represented by the notation $\Zero$ ($\mathbf 1$), respectively.  For  vectors $x$ and $y$, $\Proj(x,y)= \frac{\langle x| y\rangle}{\langle y| y\rangle} y$ is the projection of  $x$ into  $y$, where $\langle \boldsymbol \cdot | \boldsymbol \cdot \rangle$ is the standard inner product in the complex space. In addition, for  vector $x$ and  Hermitian positive semi-definite matrix $Q$, $\Norm{x}{Q}=\langle x| Q x\rangle$. The notation $\otimes$  denotes the Kronecker product.

\section{Equivariant and partially equivariant systems}\label{sec:invariance}

Many natural systems possess some form of  invariant features that  remain unchanged  under  certain types of transformations.  For example, the outcome of an election does not depend on the order in which  votes are registered;  Newton's laws are invariant  to translations and rotations; the  spectrum of an adjacency matrix of a graph is invariant to particular drawings or labels of vertices, and  the presence of an object in  an image  is invariant to spatial transformation.  Motivated by  invariant theory in mathematics~\cite{dieudonne1970invariant}  and  invariance mechanics in  physics~\cite{noether1971invariant}, which play a key role in describing  natural processes, we study an equivariant linear quadratic system, where features  have linear dynamics and quadratic cost functions.

Consider a simple linear quadratic system consisting of $n$ agents (decision makers). Let $d_x  \in \mathbb{N}$ and $d_u \in \mathbb{N}$ denote the dimension of the local state and local action  of each agent, respectively.  The  dynamics of  the  joint state  at time $t \in [0, \infty)$ is described by:
$\dot  x_t=A_t  x_t+B_t  u_t$,
where $ x_t=\VEC(x^1_t,\ldots,x^n_t) \in \mathbb{R}^{nd_x}$   and $ u_t = \VEC(u^1_t,\ldots,u^n_t) \in  \mathbb{R}^{nd_u}$. The cost function is given by: $ \Norm{x_T}{Q_T} +\int_{0}^T (\Norm{x_t}{Q_t}+ \Norm{u_t}{R_t}) dt$, where $T \in [0, \infty)$. 

Denote by $F=F(1) \circ F(2)\ldots \circ F(Y) \in \mathbb{R}^{n\times n}$ the   composition (concatenation) of  $Y$ different  transformations (filters) $F(y)$, $y \in \mathbb{N}_Y$.   Let   $F_x:=F \otimes \mathbf I_{d_x \times d_x}$ and $F_u:=F \otimes \mathbf I_{d_u \times d_u}$.   Define  $\mathbf c_t(F):=\VEC(c_t^1(F),\ldots,c_t^n(F))$ as a vector of cost functions,  where  $c_t^i(F):=\Norm{F_x x_t}{Q_t}  + \Norm{F_u u_t}{R_t}$, $ i \in \mathbb{N}_n$.

 \begin{Definition}[Equivariant linear quadratic system]\label{Def:Invariant}
A linear quadratic system is  said to be equivariant to transformation~$F$ if the following conditions hold at any time $t \in [0, \infty) $:
\begin{itemize}
\item Equivariant dynamics: $F_x \dot  x_t=A_t  F_x x_t+ B_t  F_u u_t$.
\item Equivariant cost function:  $F^\intercal F \mathbf c_t (\mathbf I_{n\times n})=\mathbf c_t(F)$. Alternatively, 
  $\TR( F_x^\intercal F_x Q_t x_t x_t^\intercal)+ \TR( F_u^\intercal F_u R_t u_t u_t^\intercal)=\Norm{F_xx_t}{Q_t}+\Norm{F_u u_t}{R_t}$.
\end{itemize}
\end{Definition}
\begin{Remark}
\emph{
The notion of equivariance is more general than that of invariance.   In particular, a function $g$ is invariant under the transformation group $\mathcal{F}$ with domain $\mathcal{X}$ if  $\forall x \in \mathcal{X}$ and $\forall f \in \mathcal{F} $,  $g(x)=g(f(x))$. On the other hand,  a function $g$ is equivariant  if  $f(g(x))=g(f(x))$. Note that the output also changes,  which is in contrast to the invariant case. 
}
\end{Remark}

\begin{Definition}[Polynomial family]
A linear quadratic system is said to be polynomial  in $F$ if  its corresponding matrices are polynomial functions of matrix  $F$. In particular, at any time  $t \in [0, \infty)$,  $A_t=\sum_{h_a=0}^{H_a} a_t(h) F_x^{h_a} $,  $B_t=\sum_{h_b=0}^{H_b} b_t(h)F_u^{h_b}$, $Q_t=\sum_{h_q=0}^{H_q} q_t(h)  F_x^{h_q}$ and $R_t=\sum_{h_r=0}^{H_r} r_t(h) F_u^{h_r}$,  where $a_t(h),b_t(h),q_t(h),r_t(h) \in \mathbb{R}$ and $H_a, H_b, H_q, H_r \in \mathbb{N} \cup \{0\}$.\footnote{If $F$ is invertible, the polynomial functions can be extended to Laurent polynomial functions including both $F$ and $F^{-1}$.}  
\end{Definition}
\begin{Proposition}\label{prop:polynomial}
Any   linear quadratic system that is polynomial in $F$ is equivariant to  $F$.
\end{Proposition}

\begin{proof}
The proof is provided in Appendix~\ref{sec:proof_prop:polynomial}.
\end{proof}

\begin{Definition}[Partially equivariant linear quadratic system]\label{Def:partial}
A linear quadratic system is said to be partially equivariant if it can be partitioned into $S \in \mathbb{N}$ distinct sub-populations, where  the dynamics and cost of agents in each sub-population  are equivariant to some transformations. 
\end{Definition}

In what follows, we present  some  equivariant linear quadratic systems, where agents are coupled through a set of linear regressions of the states and actions of agents. 
\subsection{Some structural results}
 To simplify the analysis, we  consider scalar variables in this subsection, where   $d_x=d_u=1$.  Let  $\lambda_j$ and $v_j$, $j \in \mathbb{N}_n$, denote the $j$-th  eigenvalue and eigenvector  of $F$, respectively.
\begin{Proposition}[Normal transformation]\label{thm:normal}
Suppose $F$ is a real-valued  normal matrix,  where $F^\intercal F=F F^\intercal$. The  instantaneous cost of the transformed  system is proportional to that of the original  system along each eigenvector of matrix~$F$ such that
$\Norm{F x_t}{Q_t}+\Norm{F u_t}{R_t}=\sum_{j=1}^n \langle \lambda_j| \lambda_j \rangle (\Norm{\Proj(x_t,v_j )}{Q_t}
+\Norm{\Proj(u_t,v_j)}{R_t})$.
\end{Proposition}
\begin{proof}
The proof is presented in Appendix~\ref{sec:proof_thm:normal}. 
\end{proof}
For  real-valued matrices,    symmetric,  skew-symmetric and orthogonal matrices are normal.
\begin{Proposition}[Symmetric transformation]\label{proposition:symmetric}   Given any real-valued symmetric transformation $F$,  let 
$\bar x^j_t:=\frac{1}{n}\sum_{i=1}^n \alpha^{i,j} x^i_t$ and  $\bar u^j_t:=\frac{1}{n}\sum_{i=1}^n \alpha^{i,j} u^i_t$, 
where  $\alpha^{i,j} \in \mathbb{R}$  is the $i$-th element of  the vector $\sqrt n v_j$, $i,j \in \mathbb{N}_n$. For  every agent $i \in \mathbb{N}_n$ at any time $t \in [0,T]$, the following linear quadratic system: 
\begin{equation}
\begin{cases}
\dot x^i_t= a_t x^i_t+ b_t u^i_t+ \sum_{j=1}^n \alpha^{i,j} ( \bar a^j_t \bar x^j_t +   \bar a^j_t \bar u^j_t),\\
c^i_t=\sum_{j=1}^n ( \int_{0}^T (\Norm{\bar x^j_t}{\bar q^j_t} + \Norm{\bar u_t}{\bar r^j_t})dt )+\sum_{i=1}^n (   \int_{0}^T (\Norm{x^i_t}{q_t} + \Norm{u^i_t}{r_t})dt), 
\end{cases}
\end{equation}
is equivariant to   transformation $F$, where $a_t:= a_t(0)$, $\bar a^j_t:= \sum_{h=1}^{H_a} \lambda_j^h a_t(h) $,
$b_t:= b_t(0)$,  $\bar b^j_t:= \sum_{h=1}^{H_b} \lambda_j^h b_t(h)$, $
q_t:= q_t(0) $,  $\bar q^j_t:= \sum_{h=1}^{H_q} \lambda_j^h q_t(h) $, $
r_t:= r_t(0) $, and  $\bar r^j_t:= \sum_{h=1}^{H_r} \lambda_j^h r_t(h)$.
\end{Proposition}
\begin{proof}
The proof is presented in Appendix~\ref{sec:proof_proposition:symmetric}.
\end{proof}
According to Proposition~\ref{proposition:symmetric}, each eigenvector represents a specific  feature of the transformation $F$, where for example, $\bar x_t^j$  denotes the aggregate state  of agents associated with the feature identified by the eigenvector $v_j$, $j \in \mathbb{N}_n$. In real-world applications,  it is  often useful  to  restrict attention to only a small subset of dominant features  corresponding to the largest  eigenvalues. For the case of  an arbitrary permutation matrix~$F$ (which is generally not a symmetric matrix, and  hence may admit complex eigenvalues), we show that all features become equally important,   resulting in an aggregate feature represented  by the empirical (unweighted) average.

\begin{Proposition}[Arbitrary permutation]\label{proposition:permutation}
  Let   $\bar x_t:=\frac{1}{n}\sum_{i=1}^n x^i_t$ and    $\bar u_t:=\frac{1}{n}\sum_{i=1}^n u^i_t$, $i \in \mathbb{N}_n$. For  every agent $i \in \mathbb{N}_n$ at any time $t \in [0,T]$, the  following linear quadratic system:
\begin{equation}
\begin{cases}
\dot x^i_t= a_t x^i_t+ b_t u^i_t+ \bar a_t \bar x_t +  \bar b_t \bar u_t,\\
 c^i_t=  \int_{0}^T (\Norm{\bar x_t}{\bar q_t} + \Norm{ \bar u_t}{\bar r_t})dt)
+\sum_{i=1}^n (   \int_{0}^T (\Norm{x^i_t}{q_t} + \Norm{u^i_t}{r_t})dt),
\end{cases}
\end{equation}
 is equivariant to every  permutation matrix~$F$, where $A_t=:a_t \mathbf{I}_{n \times n} + \bar a_t \mathbf{1}_{n \times n}$, $B_t=:b_t \mathbf{I}_{n \times n} + \bar b_t \mathbf{1}_{n \times n}$, $Q_t=:q_t \mathbf{I}_{n \times n} + \bar q_t \mathbf{1}_{n \times n}$, and  $R_t=:r_t \mathbf{I}_{n \times n} + \bar r_t \mathbf{1}_{n \times n}$.
\end{Proposition}
\begin{proof}
The proof is presented in Appendix~\ref{sec:proof_proposition:permutation}.
\end{proof}

\begin{Remark}
\emph{ Note that the above matrices can depend on $n$. 
}
 \end{Remark}
Inspired from  the structural results obtained in Propositions~\ref{proposition:symmetric} and~\ref{proposition:permutation},   we consider a  more general model in this paper wherein agents are partitioned into a few  heterogeneous  sub-populations and are coupled through a set of   linear regressions  of  the states and actions with  multivariate parameters.

%

\section{Problem Formulation}\label{sec:problem}

Consider a stochastic dynamic control system consisting of $n \in \mathbb{N}$ agents. The agents are  partitioned into $S \in \mathbb{N}$  sub-populations  wherein  each  sub-population $s \in \mathbb{N}_S$ contains $n(s) \in \mathbb{N}$  agents, i.e.,  $\sum_{s=1}^S n(s)=n$.  Let $x^i_t \in \mathbb{R}^{d^s_x}$ and $u^i_t \in \mathbb{R}^{d^s_u}$, respectively, denote the state and action of agent $i$ in sub-population $s \in \mathbb{N}_S$ at time $t \in [0, \infty)$,  where $d^s_x,d^s_u \in \mathbb{N}$. Let $\alpha^{i,j}(s) \in \mathbb{R}$  be the  \emph{influence factor} of agent $i \in \mathbb{N}_{n(s)}$ associated with the $j$-th feature of  sub-population~$s$,  $j \in \mathbb{N}_{f(s)}, f(s) \in \mathbb{N}$. The influence factors are orthogonal vectors in the feature space such that
\begin{equation}\label{eq:orthogonal_features}
\frac{1}{n(s)}\sum_{i=1}^{n(s)} \alpha^{i,j}(s) \alpha^{i,j'}(s)=\ID{j= j'}, \quad j,j' \in \mathbb{N}_{f(s)}.
\end{equation}
For any feature $j \in \mathbb{N}_{f(s)}$ of sub-population $s \in \mathbb{N}_S$ at time $t \in [0, \infty)$, define the following linear regressions (weighted averages) of the states and actions of agents:
\begin{equation}\label{eq:mf1}
 \bar x_t^j(s):=\frac{1}{n(s)}\sum_{i=1}^{n(s)} \alpha^{i,j}(s) x^i_t, \quad \bar u_t^j(s):=\frac{1}{n(s)}\sum_{i=1}^{n(s)} \alpha^{i,j}(s)u^i_t.
\end{equation} 
In the sequel, we  refer to  $\bar x_t^j(s)$ and $\bar u_t^j(s)$ as the $j$-th \emph{deep state} and \emph{deep action}  of sub-population $s$, respectively.

For every  $s \in \mathbb{N}_S$, let $\bar{\mathbf x}_t(s):=\VEC(\bar x^1_t(s),\ldots,\bar x^{f(s)}_t(s))$ and $\bar{\mathbf u}_t(s):=\VEC(\bar u^1_t(s),\ldots,\bar u^{f(s)}_t(s))$. Let also $\bar {\mathbf x}_t:=\VEC(\bar{\mathbf x}_t(1),\ldots,\bar{\mathbf x}_t(S))$ and  $\bar {\mathbf u}_t:=\VEC(\bar{\mathbf u}_t(1),\ldots,\bar{\mathbf u}_t(S))$.  The dynamics of agent~$i$ in sub-population $s$  is  affected by  other agents through the deep states and deep actions of all sub-populations  as follows:
\begin{equation}\label{eq:dynamics1}
d(x^i_t)=\big(A_t(s) x^i_t + B_t(s)u^i_t
+ \sum_{j=1}^{f(s)} \alpha^{i,j}(s)( \bar A_t^j(s) \bar{ \mathbf x}_t+  \bar B_t^j(s) \bar{\mathbf u}_t)\big) dt+ C_t(s) dw^i_t,
 \end{equation}
  where   $A_t(s)$,  $B_t(s)$,  $\bar A_t^j(s)$, $\bar B_t^j(s)$  and $C_t(s)$  are time-varying matrices  of  appropriate dimensions,  and     $\{w^i_t \in \mathbb{R}^{d^s_x} | t \in [0,\infty)\}$  is an  $d^s_x$-dimensional  standard Brownian motion.  Let  $ \Sigma^{w}_t(s):=C_t(s) (C_t(s))^\intercal$ for every agent $i \in \mathbb{N}_{n(s)}$ in sub-population $s \in \mathbb{N}_S$ at time $t \in [0,\infty)$.

Denote  by $\mathbf x_t$, $\mathbf u_t$ and $\mathbf w_t$, the joint state, joint action and joint noise  of  all agents at time $t \in [0,T]$, respectively.  Let $(\Omega, \mathcal{F}, \mathbb{P}; \mathcal{F}_{t })$ be a filtered probability space, where $\mathcal{F}_t$ is an increasing sigma-algebra  generated by random variables $\{\mathbf x_0, \mathbf w_t; t \in [0,T]\}$.  It is assumed  that $\{\mathbf x_0,\mathbf w_t; t  \in [0,T]\}$ are mutually independent across  time horizon.  In addition, local noises are mutually independent across agents.    It is to be noted that the initial states can be arbitrarily correlated  across agents.

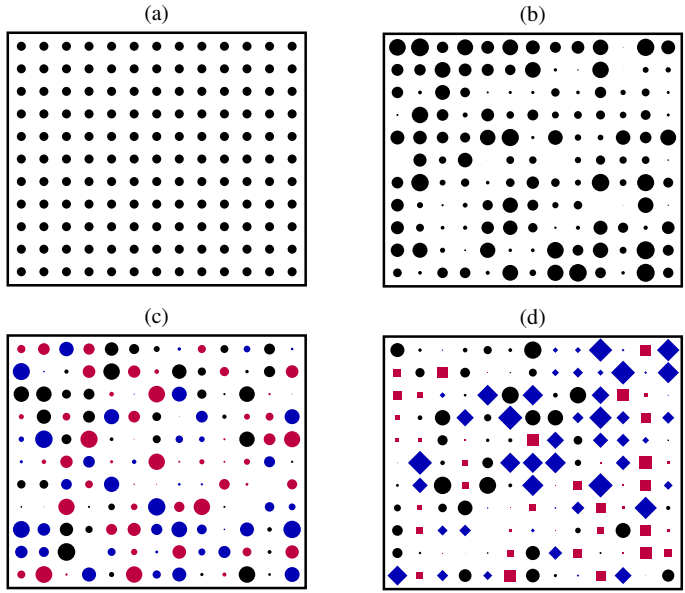
\begin{figure}
\hspace{3cm}
\scalebox{0.6}{
\begin{tikzpicture}
\pgfmathdeclarerandomlist{color}{{black}}
\foreach \x in {0,...,12}
\foreach \y in {0,...,10}
\pgfmathrandomitem{\c}{color}
\fill[color=\c] (.5*\x,.5*\y) circle  [radius=.1];
\draw [black,ultra thick] (-.3,-.3) rectangle (6.3,5.3);
\draw[font=\Large](3cm,5.7cm) node {(a)}; 
\end{tikzpicture}}

\vspace{-3.8cm}
\hspace{8cm}
\scalebox{0.6}{
\begin{tikzpicture}
\pgfmathdeclarerandomlist{color}{{black}}
\foreach \x in {0,...,12}
\foreach \y in {0,...,10}
\pgfmathrandomitem{\c}{color}
\pgfmathsetmacro{\opacVal}{rand*0.2}
\fill[color=\c] (.5*\x,.5*\y) circle  [radius=\opacVal];
\draw [black,ultra thick] (-.3,-.3) rectangle (6.3,5.3);
\draw[font=\Large](3cm,5.7cm) node {(b)}; 
\end{tikzpicture}}

\vspace{.2cm}
\hspace{3cm}
\scalebox{0.6}{
\begin{tikzpicture}{x=1cm,y=-10cm}
\pgfmathdeclarerandomlist{color}{{purple}{black}{blue}}
\foreach \x in {0,...,12}
\foreach \y in {0,...,10}
\pgfmathrandomitem{\c}{color}
\pgfmathsetmacro{\opacVal}{rand*0.2}
\fill[color=\c] (.5*\x,.5*\y) circle  [radius=\opacVal];
\draw [black,ultra thick] (-.3,-.3) rectangle (6.3,5.3);
\draw[font=\Large](3cm,5.7cm) node {(c)}; 
\end{tikzpicture}}

\vspace{-3.8cm}
\hspace{8cm}
\scalebox{0.6}{
\begin{tikzpicture}
\pgfmathdeclarerandomlist{shape}{{rectangle} {circle}{diamond}}
\pgfmathdeclarerandomlist{color}{{purple}{black}{blue}}
\foreach \x in {0,...,12}
\foreach \y in {0,...,10}
\pgfmathrandomitem{\d}{shape}
\pgfmathrandomitem{\c}{color}
\pgfmathsetmacro{\opacVal}{rand*1.2}
\draw[ black,scale=1] (.5*\x,.5*\y) node[\d, fill=\c, scale=\opacVal]{};
\draw [black,ultra thick] (-.3,-.3) rectangle (6.3,5.3);
\draw[font=\Large](3cm,5.7cm) node {(d)}; 
\end{tikzpicture}} 
\caption{To  visualize the complexity  of the  interaction between agents,  the roles of  influence  factor, feature, and sub-population are displayed by the size, color and shape of the agents, respectively. In plot (a), agents  have homogeneous weights with single feature and single sub-population. In plot (b), agents have heterogeneous weights  with single feature and single sub-population. In plot (c), agents have heterogeneous  weights with multiple features and single sub-population. In plot (d), agents have  heterogeneous  weights with multiple  features and multiple sub-populations.}
\end{figure}

Let  $r^i_t \in \mathbb{R}^{d^s_x}$ be the desired operating point of agent $i$ in sub-population $s$  at  time $t \in [0,\infty)$.
 The cost of agent $i$ in sub-population $s \in \mathbb{N}_S$ over the the finite horizon specified by~$T$ is defined as:
 \begin{equation}\label{eq:cost1}
J^{i}_T= \Norm{x^i_T-r^i_T}{Q_T(s)}  + \Norm{\bar{\mathbf x}_T}{\bar Q_T(s)} 
      +   \int_{0}^T \big(\Norm{x^i_t-r^i_t}{Q_t(s)} + \Norm{u^i_t}{R_t(s)}     + \Norm{\bar{\mathbf x}_t}{\bar Q_t(s)} +  \Norm{\bar{\mathbf u}_t}{\bar R_t(s)}   \big)dt,
\end{equation} 
where   matrices $Q_t(s)$, $R_t(s)$,   $\bar Q_t(s)$ and  $\bar R_t(s)$, $t \in [0,T]$,  are symmetric with appropriate dimensions.  To have a well-posed problem,   it is assumed  that the expectation and covariance matrices  of initial states as well as local noises are uniformly bounded in time  and number of agents, and that  the set of admissible control actions are  adapted to the filtration $\mathcal{F}_t$ and  square integrable   for  all agents, i.e.   $\Exp{\sum_{t=1}^T (u^i_t)^\intercal u^i_t} < \infty$, $\forall i \in \mathbb{N}_n(s)$, $\forall s \in \mathbb{N}_S$.

Let $\mu(s) \in  (0,\infty)$ denote the influence factor of  sub-population $s \in \mathbb{N}_S$  among all sub-populations.
 As an example, the influence factor of sub-population $s$ may be defined based on its size in a population  or   its topological  configuration in a network. In the sequel, we refer to $\alpha^{i,j}(s)$ as  the microscopic influence factor (that determines the influence of an individual agent $i$ in  sub-population $s$) and to $\mu(s)$ as the macroscopic  influence factor (that indicates the influence of sub-population~$s$ on the entire population). Let
\begin{equation}\label{eq:total_cost}
\bar J^n_T:=\sum_{s=1}^S  \frac{\mu(s)}{n(s)} \sum_{i=1}^{n(s)} J^{i}_T,
\end{equation}
where  superscript $n$ refers to the dependency with respect to the number of agents.

\begin{Definition}[Weakly coupled agents]\label{def:weakly}
The agents are said to be weakly coupled in dynamics if the coupling term in~\eqref{eq:dynamics1} at time $t\in [0,T]$ reduces to the following form:
$\sum_{j=1}^{f(s)} \alpha^{i,j}(s)(\bar A^j_t(s)\bar x^j_t(s)+\bar B^j_t(s)\bar u^j_t(s))$.
Similarly, the  agents are said to be weakly coupled in cost function if the coupling term in~\eqref{eq:total_cost}  at time $t\in [0,T]$ is:
$\sum_{s=1}^S \mu(s)\sum_{j=1}^{f(s)} \big(\Norm{\bar x^j_t(s)}{\bar Q^j_t(s)}+ \Norm{\bar u^j_t(s)}{\bar R^j_t(s)}\big)$.
\end{Definition}
The weakly coupled agents  emerge in various applications such as  those with equivalent structure  in Propositions~\ref{proposition:symmetric} and~\ref{proposition:permutation}.

\begin{Remark}[Partially exchangeable systems]
\emph{It is shown in~\cite[Chapter 2]{arabneydi2016new} that any  partially exchangeable LQ system\footnote{A system (population) is said to be  partially exchangeable if exchanging  any pair of  agents in each sub-population does not affect the dynamics and cost function of the entire system.} can be equivalently represented as an LQ deep structured system, where the number of orthogonal  features in each sub-population  is one and the agents   are equally important  within their own sub-populations, i.e.  $\alpha^{i,j}=\frac{1}{n(s)}$, $\forall i \in \mathbb{N}_{n(s)}, j=1$).}
\end{Remark}

\begin{figure*}
\hspace{-8.4cm}
\scalebox{0.7}{
\begin{tikzpicture}[x=1.5cm, y=1.2cm, >=stealth]
  \draw[font=\large](0,3cm) node {Layer $t$} ;
\draw[font=\large](2.9cm,3cm) node {Layer $t+\delta t$} ;
\draw[font=\large](10cm,-5cm) node {Sub-population $  s$} ;
\draw[black, thick, dashed] (-1.2,2.1) rectangle (14.5,-3.8);

\foreach \m/\l [count=\y] in {1,missing,2,missing,3}
  \node [every neuron/.try, neuron \m/.try] (input-\m) at (0,2-\y) {};

\foreach \m/\l [count=\y] in {1,missing,2,missing,3}
  \node [every neuron/.try, neuron \m/.try ] (output-\m) at (2,2-\y) {};

  \draw [<-] (input-1) -- ++(-1,0)
    node [above,midway] {\hspace{3.5cm}\large $ x^1_t$}
        node [below,midway] {\hspace{3.5cm}\large $u^1_t$};
     \draw [<-] (input-2) -- ++(-1,0)
    node [above,midway] {\hspace{3.5cm} \large $x^i_t$}
        node [below,midway] {\hspace{3.5cm} \large $u^i_t$};
        
          \draw [<-] (input-3) -- ++(-1,0)
    node [above,midway] {\hspace{3.8cm} \large $x^{n(s)}_t$}
        node [below,midway] {\hspace{3.8cm} \large $u^{n(s)}_t$};

  \draw [->]  (output-1) -- ++(1,0)  node[above, midway]  
   {\hspace*{15.5cm} \large $x^1_{t+\delta t}=x^1_t+\int_{t}^{t+\delta t} \phi^{ s}_\tau(x^1_\tau,u^1_\tau, \textcolor{red}{ \{\bar{x}^1_\tau(s),\ldots,  \bar{x}^{f( s)}_\tau(s)\}_{{ s=1}}^{ S}, \{\bar{u}^1_\tau(s),\ldots,\bar u_\tau^{f( s)}(s)\}_{{s=1}}^{{S}}})\delta \tau+ \int_{t}^{t+\delta t}  dw^1_\tau,$}
     node [below,midway] {\hspace{2.8cm} \large $u^1_{t+\delta t} =g^1_{t+ \delta t}( I^1_{t+\delta t})$,};   
      \draw [->]  (output-2) -- ++(1,0)  node[above, midway]    {\hspace*{15.5cm} \large $x^i_{t+\delta t}=x^i_t+\int_{t}^{t+\delta t} \phi^s_\tau(x^i_\tau,u^i_\tau, \textcolor{red}{ \{\bar{x}^1_\tau(s),\ldots,  \bar{x}^{f(s)}_\tau(s)\}_{s=1}^S, \{\bar{u}^1_\tau(s),\ldots,\bar u_\tau^{f(s)}(s)\}_{s=1}^S})\delta \tau+ \int_{t}^{t+\delta t}  dw^i_\tau,$}
     node [below,midway] {\hspace{2.8cm} \large $u^i_{t+\delta t} =g^i_{t+ \delta t}( I^i_{t+\delta t})$,}; 
    
      \draw [->]  (output-3) -- ++(1,0)  node[above, midway]    {\hspace*{17cm} \large  $x^{n(s)}_{t+\delta t}=x^{n(s)}_t+\int_{t}^{t+\delta t} \phi^s_\tau(x^{n(s)}_\tau,u^{n(s)}_\tau, \textcolor{red}{ \{\bar{x}^1_\tau(s),\ldots,  \bar{x}^{f(s)}_\tau(s)\}_{s=1}^S, \{\bar{u}^1_\tau(s),\ldots,\bar u_\tau^{f(s)}(s)\}_{s=1}^S})\delta \tau+ \int_{t}^{t+\delta t}  dw^{n(s)}_\tau,$}
     node [below,midway] {\hspace{2.6cm} \large $u^{n(s)}_{t+\delta t} =g^{n(s)}_{t+ \delta t}( I^{n(s)}_{t+\delta t})$.};    
    \foreach \i in {1,...,3}
  \foreach \j in {1,...,3}
    \draw [->] (input-\i) -- (output-\j);    
\end{tikzpicture}
}
\caption{The interaction between agents  in a deep structured teams is similar to that between the neurons in a deep feed-forward neural network, where $\phi^s_t$ is an affine function in this article, and $I^i_t$ denotes the information set of agent $i \in \mathbb{N}_{n(s)}$ in sub-population $s \in \mathbb{N}_S$ at time $t \in [0,T]$.}
\end{figure*}
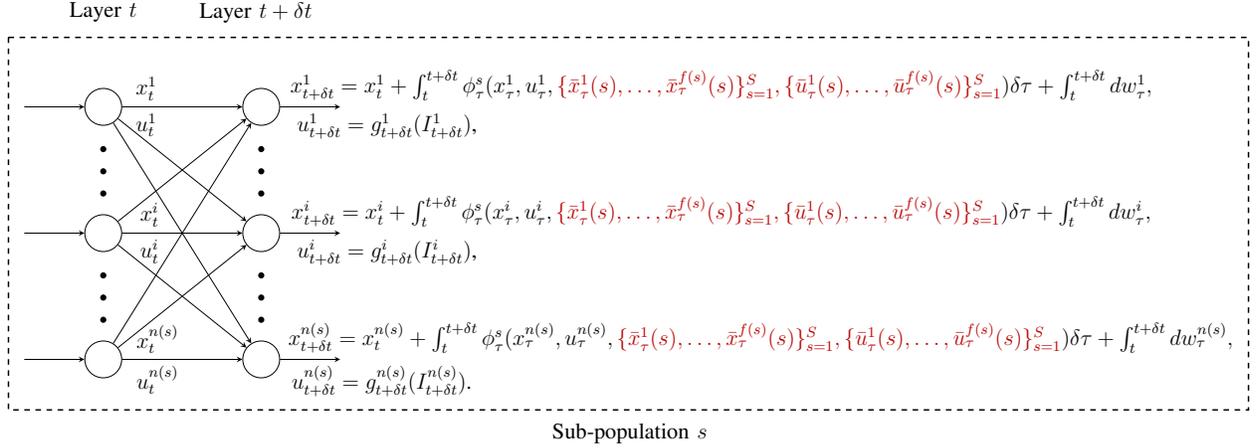

\subsection{Information structure}
Similar to~\cite{Jalal2019MFT}, in this paper, we consider two non-classical information structures. The first one is called \emph{deep-state sharing} (DSS), where any agent $i \in \mathbb{N}_{n(s)}$ of sub-population $s \in \mathbb{N}_S$ observes its local state $x^i_t$ as well as the joint deep state $\bar {\mathbf x}_t$, i.e.,
\begin{equation}
u^i_t=g^i_t(x^i_{t}, \bar{ \mathbf x}_{t}), \tag{DSS}
\end{equation}
where $g^i_t$ is called  the control law of agent $i$ at time $t \in [0,T]$, which is a measurable function adapted  to the $\sigma$-algebra generated by $\{x^i_0,\bar{ \mathbf x}_{0}, w^i_t,\bar{ \mathbf w}_{t}; t \in [0,T]\}$.  In practice,  deep states can be shared among agents in various ways. For example,   in a stock market, the total amount of shares, sales and  trades are often announced publicly, and hence are available  to buyers, sellers and traders.   Another example is  cloud-based applications, where a  central server collects and shares the deep states.  Alternatively, the deep states can be computed in a distributed manner by local communications among agents using a  consensus algorithms under some  assumptions, e.g. in a swarm of robots, where the  control process is often slower than  the communication  process.  The second information  structure is called~\emph{partial deep-state sharing} (PDSS), where each agent $i$  at time $t \in [0,T]$ observes its local state as well as a subset $\mathcal{O} \subseteq \mathbb{N}_S$ of the deep states, i.e.,
\begin{equation}
u^i_t=g^i_t(x^i_{t}, (\bar{\mathbf x}_t(s))_{s \in \mathcal{O}}), \tag{PDSS}
\end{equation}
where $g^i_t$ is a measurable function adapted to the $\sigma$-algebra generated by $\{x^i_0, (\bar{\mathbf x}_0(s))_{s \in \mathcal{O}}, w^i_t,(\bar{\mathbf w}_t(s))_{s \in \mathcal{O}}; t \in [0,T]\}$.
 When a sub-population is large,  collecting  its deep states  and sharing them with other  agents may  be infeasible. In this case, PDSS  structure is desirable as it does not contain the deep states of large sub-populations. Note that deep-state sharing and no-sharing  information structures are special cases of PDSS, associated with $\mathcal{O}=\mathbb{N}_S$ and $\mathcal{O}=\emptyset$, respectively. It is worth highlighting  that the privacy of  agents are respected  under the above  structures because  their local states  are not shared  with others.   In the sequel, we refer to $\mathbf g:=\{\{\{g^i_t\}_{i=1}^{n(s)}\}_{s=1}^S; t \in [0,T]\}$
 as the \emph{strategy} of the system,  which is the collection of  control laws across  control horizon.
\subsection{Problem statement}\label{sec:problem_statement} 
 For any \emph{risk parameter} $\lambda \in (0,\infty)$, define the following risk-sensitive cost function:
\begin{equation}\label{eq:total_cost_risk}
\tilde J^n_T(\mathbf g):=\frac{1}{\lambda} \log \Exp{e^{\lambda J_T}},
\end{equation}
where argument $\mathbf g$ refers to the dependency of the right-hand side of~\eqref{eq:total_cost_risk} to the strategy of the system.
 By twice  using  the  Taylor series expansion   and keeping the first two dominant terms, it is shown in~\cite{jacobson1973optimal,whittle1981risk}  that   for small~$\lambda$,
 $\tilde J^n_T \approx \Exp{\bar J^n_T} +\frac{\lambda }{2} \VAR(\bar J^n_T)$.
Therefore, the risk-sensitive cost function converges to the risk-neutral one  as $\lambda$ goes to zero, i.e.,  $\lim_{\lambda \rightarrow 0} \tilde J^n_T=\Exp{\bar J^n_T}$. An immediate implication  of the above approximation is that  the risk-factor $\lambda$  balances the trade off between optimality (where $\lambda \rightarrow 0$) and robustness (where robustness is defined in terms of minimum variance). In other words,  the risk-sensitive cost function takes into account not only the optimality but also the robustness of the design.

The problems investigated in this paper  are defined below.

\begin{Problem}\label{prob1}
For deep-state sharing information structure, find an optimal strategy $\mathbf g^\ast$  s.t. $\tilde J^n_T(\mathbf g^\ast) \leq \tilde J^n_T(\mathbf g)$, $\forall \mathbf g$.
\end{Problem}

\begin{Remark}
\emph{For the special case of single sub-population (i.e. $S=1$), single agent (i.e. $n(s)=1$, $s \in \mathbb{N}_S$) and single feature (i.e. $f(s)=1$, $s \in \mathbb{N}_S$), Problem~\ref{prob1}  reduces to  the classical   control problem in~\cite{jacobson1973optimal,whittle1981risk,bacsar2008h}. In this light,    deep structured teams with DSS structure  may be viewed as a generalization   of the  single-agent (classical) control problems.}
\end{Remark}

Let $n^\star$ be the smallest sub-population whose deep state is not shared, i.e. $n^\star:=\min_{s \notin \mathcal{O}} n(s)$.

\begin{Problem}\label{prob2}
For  partial deep-state sharing information structure, find a sub-optimal strategy $\hat{\mathbf g}$ such that  $\tilde J^n_T(\hat{\mathbf g}) \leq \tilde J^n_T(\mathbf g)+\varepsilon(n^\ast)$,  $\forall \mathbf g$, where $\lim_{n^\ast \rightarrow \infty} \varepsilon(n^\ast)=0$.
\end{Problem}

\subsection{Main challenges and contributions}
 There are two main challenges in order to solve Problems~\ref{prob1} and~\ref{prob2}. The first one is the \emph{curse of dimensionality} because  the coupling between agents in  dynamics~\eqref{eq:dynamics1} and cost function~\eqref{eq:total_cost_risk} are fully dense with no sparsity,  where  computational complexity increases with the number of agents (i.e.,  intractable for large-scale problems). The second challenge is the \emph{non-classical (decentralized) nature of the information structure}, where the resultant optimization is non-convex~\cite{Witsenhausen1968Counterexample}  and the dynamic programming decomposition is not applicable due to the fact that  the agents cannot find a  sufficient statistic for  sequential decomposition. The above  challenges are exacerbated upon noting that the certainty equivalence principle does not hold for the risk-sensitive cost function. The main contributions of this article are outlined below.
\begin{enumerate}

  \item  We  study the notions of equivariance and partial equivariance in linear quadratic models, and  inspired by the obtained  structural results (Propositions~\ref{prop:polynomial}--\ref{proposition:permutation}) as well as recent architectural developments in deep learning, we  introduce LQ  deep structured teams. We propose  a gauge transformation technique to identify a closed-form tractable representation of the optimal solution for any  arbitrary number of agents under DSS structure, where the feedback gains are computed by a  deep Riccati equation whose dimension is independent of the number of agents in each sub-population (Theorem~\ref{thm:mfs}).

\item We  show that the effect of  the risk factor has a reverse relationship with  the number of agents, implying that  the risk-sensitive  solution converges to the risk-neutral one as the number of agents goes to infinity (Corollary~\ref{cor:risk}).   In addition,  we introduce extra optimization factors for the risk-neutral case (Corollary~\ref{cor:optimization}). 

 \item We establish an explicit  connection with a deep feed-forward neural network and demonstrate its usefulness in advancing our knowledge of   deep learning methods as well as designing  more efficient  networks.
 
 \item We propose two  asymptotically optimal  strategies under PDSS structure by  introducing two Kalman-like filters, one based on the finite-population model and one based on the infinite-population one (Theorem~\ref{thm:pmfs}). Furthermore,  we extend our   main results to the   infinite-horizon cost function (Theorem~\ref{cor:infinite}). 
\end{enumerate}

\section{Main Results with deep-state sharing} \label{sec:main}
 
Define the following matrices for every sub-population $s \in \mathbb{N}_S$ at any time $t \in [0,T]$:
\begin{align}
\bar {\mathbf A}_t(s)&:=[\mathbf 0_{f(s)d^s_x\times f(1)d^1_x},\ldots,\tilde{\mathbf A}_t(s),\ldots,\mathbf 0_{f(s)d^s_x\times f(S)d^S_x}]+\ROW(\bar A^1_t(s),\ldots,\bar A^{f(s)}_t(s)),\\
\bar {\mathbf B}_t(s)&:=[\mathbf 0_{f(s)d^s_u\times f(1)d^1_u},\ldots,\tilde{\mathbf B}_t(s),\ldots,\mathbf 0_{f(s)d^s_u\times f(S)d^S_u}] +\ROW(\bar B^1_t(s),\ldots,\bar B^{f(s)}_t(s)),
\end{align}
where 
\begin{align}
\tilde {\mathbf A}_t(s)&:=\DIAG(A_t(s))_{f(s)},\quad \bar {\mathbf Q}_t(s):=\DIAG(Q_t(s))_{f(s)},\\
\tilde {\mathbf B}_t(s)&:=\DIAG(B_t(s))_{f(s)},\quad \bar {\mathbf R}_t(s):= \DIAG(R_t(s))_{f(s)},\\
\bar {\mathbf C}_t(s)&:=\DIAG(C_t(s))_{f(s)},\quad \bar{\boldsymbol \Sigma}^w_t(s):=\DIAG(\Sigma^w_t(s))_{f(s)}.
\end{align}
In addition, define the followings for the entire population:
\begin{align}\label{eq:augmented_variables}
\bar{\mathbf A}_t  &:=\ROW(\bar{\mathbf A}_t(1),\ldots, \bar{\mathbf A}_t(S)), 
\bar{\mathbf B}_t:=\ROW(\bar{\mathbf B}_t(1),\ldots, \bar{\mathbf B}_t(S)),  \nonumber  \\
\bar{\mathbf Q}_t&:=\DIAG(\mu(1)\bar{\mathbf Q}_t(1),\ldots,\mu(S)\bar{\mathbf Q}_t(S))+\sum_{s=1}^S \mu(s) \bar Q_t(s), \nonumber  \\
\bar{\mathbf R}_t&:=\DIAG(\mu(1)\bar{\mathbf R}_t(1),\ldots,\mu(S)\bar{\mathbf R}_t(S))+\sum_{s=1}^S \mu(s) \bar R_t(s), \nonumber  \\
\bar{\boldsymbol  \Sigma}^w_t&:= \DIAG(\frac{1}{n(1)} \bar{\boldsymbol \Sigma}^{w}_t(1),\ldots, \frac{1}{n(S)} \bar{\boldsymbol \Sigma}^{w}_t(S)), \nonumber  \\
\bar{\mathbf C}_t&:=\DIAG(\bar{\mathbf C}_t(1),\ldots,\bar{\mathbf C}_t(S)).
\end{align}
  To derive our main results, we  make the following  standard assumption on the model.

\begin{Assumption}\label{assumption1}
For any $t \in [0,T]$ and $s \in \mathbb{N}_S$,
\begin{itemize}
\item[(a)] $Q_t(s)$ and $\bar{\mathbf Q}_t$ are positive semi-definite,  and  $R_t(s)$ and $\bar{\mathbf R}_t$ are positive definite, 

\item[(b)]  $B_t(s) (R_t(s))^{-1} B_t(s){}^\intercal-  2\lambda  \frac{\mu(s)}{n(s)} \Sigma^{w}_t(s)$ and   $\bar{\mathbf B}_t \bar{\mathbf R}_t^{-1} \bar{\mathbf B}_t^\intercal - 2\lambda \bar{\boldsymbol \Sigma}^w_t$  are positive definite. 
\end{itemize}
\end{Assumption}
\begin{Remark}
\emph{
 Assumption~\ref{assumption1}.(a)  is a standard convexity condition to ensure that  the optimal  solution exists and is  unique.   Assumption~\ref{assumption1}.(b) is also a standard  positive-definiteness condition arising in the risk-sensitive linear quadratic problems. It is worth highlighting that when  matrices in the dynamics~\eqref{eq:dynamics1} and~cost functions~\eqref{eq:cost1} are independent of the size of sub-populations,  $n(s), s \in \mathbb{N}_S$, the positive-definiteness condition in Assumption~\ref{assumption1}.(b) gets  relaxed as $\lambda \rightarrow 0$ and/or $n(s) \rightarrow \infty, \forall s \in \mathbb{N}_S$, where the negative terms vanish asymptotically.
 }
\end{Remark}

\subsection{Gauge transformation and HJB equation}

In invariant mechanics,  a gauge transformation, upon existence, is a powerful tool  for  the analysis of  invariant systems. In simple words, a gauge transformation  manipulates  the degrees  of freedom of  an invariant system without changing  its structure~\cite{moriyasu1983elementary}. In this paper, we use  a gauge transformation   introduced in~\cite[Appendix A.2]{arabneydi2016new} to simplify our analysis. Define the following  auxiliary variables  for  any agent $i \in \mathbb{N}_{n(s)}$ in  sub-population $s \in \mathbb{N}_S$ at  time $t \in [0,T]$: 
\begin{align}\label{eq:gauge}
 \Delta x^i_t&:=x^i_t - \sum_{j=1}^{f(s)} \alpha^{i,j}(s) \bar x_t^j(s),  \quad  \Delta u^i_t:=u^i_t - \sum_{j=1}^{f(s)} \alpha^{i,j}(s) \bar u_t^j(s), \nonumber \\
   \Delta w^i_t&:=w^i_t - \sum_{j=1}^{f(s)} \alpha^{i,j}(s) \bar w_t^j(s), \quad 
      \Delta r^i_t:=r^i_t - \sum_{j=1}^{f(s)} \alpha^{i,j}(s) \bar r_t^j(s),
\end{align} 
where 
\begin{equation}\label{eq:bar_aux}
\bar w^j_t(s):=\frac{1}{n(s)}\sum_{i=1}^{n(s)} \alpha^{i,j}(s) w^i_t,\quad \bar r_t^j(s):=\frac{1}{n(s)}\sum_{i=1}^{n(s)} \alpha^{i,j}(s) r^i_t.
\end{equation}
For every $s \in \mathbb{N}_S$, let  $\bar{\mathbf w}_t(s):=\VEC(\bar w_t^1(s),\ldots,\bar w_t^{f(s)}(s))$ and $\bar{\mathbf r}_t(s):=\VEC(\bar r_t^1(s),\ldots,\bar r_t^{f(s)}(s))$. Let also $\bar{\mathbf w}_t:=\VEC(\bar{\mathbf w}_t(1),\ldots,\bar{\mathbf w}_t(S))$ and $\bar{\mathbf r}_t:=\VEC(\bar{\mathbf r}_t(1),\ldots,\bar{\mathbf r}_t(S))$.
\begin{Lemma}\label{lemma:linear_dep}
The gauge transformation introduces the following linear dependencies for any  $j \in \mathbb{N}_{f(s)}$,  $s \in \mathbb{N}_S$ and $t \in [0,T]$:
$\sum_{i=1}^{n(s)} \alpha^{i,j}(s) \Delta x^i_t=\mathbf 0_{d^s_x \times 1}$,  $\sum_{i=1}^{n(s)} \alpha^{i,j}(s) \Delta u^i_t=\mathbf 0_{d^s_u \times 1}$,
$\sum_{i=1}^{n(s)} \alpha^{i,j}(s) \Delta w^i_t=\mathbf 0_{d^s_x \times 1}$, $ \sum_{i=1}^{n(s)} \alpha^{i,j}(s) \Delta r^i_t=\mathbf 0_{d^s_x \times 1}$.
\end{Lemma}
\begin{proof}
The proof follows directly from~\eqref{eq:orthogonal_features},~\eqref{eq:mf1} and~\eqref{eq:gauge}.
\end{proof}
\begin{Lemma}\label{lemma:orthogonal}
The gauge transformation  imposes the following orthogonal relations for any   $j \in \mathbb{N}_{f(s)}$, $s \in \mathbb{N}_S$ and  $t \in [0,T]$:
$\sum_{i=1}^{n(s)} \alpha^{i,j}(s) (\Delta x^i_t - \Delta r^i_t)^\intercal Q_t(s) (\bar x^j_t(s) - \bar r^j_t(s))=0$ and 
$\sum_{i=1}^{n(s)} \alpha^{i,j}(s) (\Delta u^i_t)^\intercal R_t(s) \bar u^j_t(s)=0$.

\end{Lemma}
\begin{proof}
The proof follows directly from~\eqref{eq:orthogonal_features},~\eqref{eq:mf1} and~\eqref{eq:gauge}.
\end{proof}
\begin{Lemma}\label{lemma:correlation}
Given any sub-population $s \in \mathbb{N}_S$,  the gauge transformation  induces the following correlations for any  $i,i' \in \mathbb{N}_{n(s)}$, $j,j' \in \mathbb{N}_{f(s)}$ and $t \in [0,T]$:
$\Exp{\Delta w^i_t) (\Delta w^i_t)^\intercal}=t  (1 -\frac{1}{n(s)} \sum_{j=1}^{f(s)} (\alpha^{i,j}(s))^2)  \mathbf{I}_{d^s_x \times d^s_x}$, 
 $\Exp{\Delta w^i_t (\bar w^j_t(s))^\intercal}=\mathbf{0}_{d^s_x \times d^s_x}$,
 $\Exp{\bar w^j_t(s) (\bar w^j_t(s))^\intercal}=\frac{t}{n(s)}\mathbf{I}_{d^s_x \times d^s_x}$,  $
\Exp{\bar w^j_t(s) (\bar w^{j'}_t(s))^\intercal}=\mathbf{0}_{d^s_x \times d^s_x}$,  and $\Exp{\Delta w^i_t(\Delta w^{i'}_t)^\intercal}=t\frac{1}{n(s)}\sum_{j=1}^{f(s)} \alpha^{i,j}(s) \alpha^{i',j}(s) \mathbf{I}_{d^s_x \times d^s_x}$.
\end{Lemma}
\begin{proof}
The proof follows  from   equations~\eqref{eq:orthogonal_features},~\eqref{eq:mf1} and~\eqref{eq:gauge}, and the fact that (a) local noises are i.i.d. $d^s_x$-dimensional standard Brownian motions in each sub-population $s \in \mathbb{N}_S$; (b) the covariance matrix of a summation of independent random vectors is equal to the summation of their covariance matrices, and  (c) local noises have zero mean.
\end{proof}

 From  Lemma~\ref{lemma:linear_dep} and equations~\eqref{eq:mf1},~\eqref{eq:dynamics1},~\eqref{eq:augmented_variables} and~\eqref{eq:gauge},  the dynamics of the transformed variables can be described  by:
\begin{align}\label{eq:dynamics_mean-field_1}
&d(\begin{bmatrix}
\Delta x^i_t\\
\bar{\mathbf x}_t
\end{bmatrix})
=\big(\DIAG(A_t(s),\bar{\mathbf A}_t) \begin{bmatrix}
\Delta x^i_t\\
\bar{\mathbf x}_t
\end{bmatrix}  +
\DIAG(B_t(s),\bar{\mathbf B}_t) \begin{bmatrix}
\Delta u^i_t\\
\bar{\mathbf u}_t
\end{bmatrix} 
\big)dt
+
\DIAG(C_t(s),\bar{\mathbf C}_t) d(\begin{bmatrix}
\Delta w^i_t\\
\bar{\mathbf w}_t
\end{bmatrix}).
\end{align}
 From Lemma~\ref{lemma:orthogonal} and after some algebraic manipulations,  cost function~\eqref{eq:total_cost} from time~$t$ up to $T$  can be expressed by
 \begin{equation}\label{eq:decomposition_cost}
\bar J^n_{t:T}:= \int_{\tau=t}^T (\bar{\mathbf L}_\tau+ \sum_{s=1}^S \frac{\mu(s)}{n(s)} \sum_{i=1}^{n(s)} \Delta L^i_\tau) d\tau,
\end{equation}
where
$\Delta L^i_t := \Norm{\Delta x^i_t - \Delta r^i_t}{Q_t(s)} + \Norm{ \Delta u^i_t}{R_t(s)}$ and 
$\mathbf L_t := \Norm{\bar{\mathbf x}_t - \bar{\mathbf r}_t}{ \DIAG(\mu(1)Q_t(1),\ldots,\mu(S)Q_t(S))} 
 + \Norm{\bar{\mathbf x}_t}{\sum_{s=1}^S \mu(s) \bar Q_t(s)} 
+ \Norm{\bar{ \mathbf u}_t}{\bar{\mathbf R}_t}$.
 Denote by  $ \mathbf y_t:=\VEC(\VEC(\VEC(\Delta x^i_t)_{i=1}^{n(s)})_{s=1}^S, \bar{\mathbf x}_t)$ and  $ \mathbf v_t:= \VEC(\VEC(\VEC(\Delta u^i_t)_{i=1}^{n(s)})_{s=1}^S, \bar{\mathbf u}_t)$  the centralized state and action of the transformed system  at time $t \in [0,T]$, respectively. Suppose for now that  $\mathbf y_t$   is known to all agents.  It will be shown later that the  optimal centralized solution can be implemented under DSS structure.  Define a real-valued function $\psi_t$ at time $t \in [0,T]$ such that
\begin{equation}\label{eq:def_psi}
\psi_t(\mathbf y):= \Exp{e^{ \lambda \bar J^n_{t:T}} \mid \mathbf y_t=\mathbf y},
\end{equation}
where  $ \tilde J^n_T= \frac{1}{\lambda}  \Exp{\log \psi_0(\mathbf y_0)}$ according to~\eqref{eq:total_cost_risk},~\eqref{eq:decomposition_cost} and~\eqref{eq:def_psi}.  Since $\frac{1}{\lambda} \Exp{ \log(\boldsymbol \cdot )}$ is a strictly increasing function,  any strategy  that minimizes $\psi_0$ will also minimize $\tilde J^n_T$. Therefore, we write the  HJB equation for the cost function~\eqref{eq:def_psi} as follows:
\begin{align}\label{eq:HJB}
&-\frac{\partial \psi_t(\mathbf y_t)}{\partial t}= \inf_{\mathbf v_t} \Big[ 
\lambda \psi_t(\mathbf y_t) (\mathbf L_t + \sum_{s=1}^S  \frac{\mu(s)}{n(s)} \sum_{i=1}^{n(s)} \Delta L^i_t) +\sum_{s=1}^S  \sum_{i=1}^{n(s)} \nabla_{\Delta x^i_t} \psi_t(\mathbf y_t)^\intercal (A_t(s) \Delta x^i_t+B_t(s) \Delta u^i_t) \nonumber \\
&+ \nabla_{\bar{\mathbf x}_t} \psi_t(\mathbf y_t)^\intercal (\bar{\mathbf A}_t \bar{\mathbf x}_t + \bar{\mathbf B}_t \bar{\mathbf u}_t) + \frac{1}{2}\TR\Big( \sum_{s=1}^S \Big\{ \sum_{i=1}^{n(s)} \sum_{i'=1}^{n(s)} \nabla_{\Delta x^i_t \Delta x^{i'}_t} \psi_t(\mathbf y_t)  \Sigma^{w}_t(s) (\frac{d}{dt}(\Exp{(\Delta w^i_t)^\intercal \Delta w^{i'}_t})) \nonumber \\
&+2\sum_{i=1}^{n(s)} \sum_{j=1}^{f(s)} \nabla_{\Delta x^i_t \bar x^j_t(s)} \psi_t(\mathbf y_t) \Sigma^{w}_t(s)(\frac{d}{dt}(\Exp{(\Delta w^i_t)^\intercal \bar w^j_t(s)}))
 +\sum_{j=1}^{f(s)}  \nabla_{\bar x^j_t(s) \bar x^j_t(s)} \psi_t(\mathbf y_t) \Sigma^{w}_t(s) (\frac{d}{dt}(\Exp{(\bar w_t^j(s))^\intercal \bar w_t^j(s)})))\Big\}\Big) \Big],
\end{align}
where  the cross terms associated with  $\Exp{(\bar w_t^{j})^\intercal \bar w_t^{j'}}=0$, $j \in \mathbb{N}_{f(s)}, j' \in \mathbb{N}_{f(s')}$, $s,s' \in  \mathbb{N}_S$  do not appear in~\eqref{eq:HJB}.

We propose the following ansatz at any time $t \in [0,T]$:
\begin{equation}\label{eq:ansatz_value}
\begin{cases}
\psi_t(\mathbf y_t)=:\bar \psi_t(\bar{\mathbf x}_t) \prod_{s=1}^S\prod_{i=1}^{n(s)}  \psi^i_t(s)(\Delta x^i_t),  \\
\bar \psi_t(\bar{\mathbf x}_t)=: \boldsymbol   \eta_t e^{\lambda(\Norm{\bar{\mathbf x}_t}{\bar{\mathbf P}_t} + \mathbf h_t)},  \\
\psi^i_t(s)(\Delta x^i_t)=: \eta_t(s) e^{\lambda (\Norm{\Delta x^i_t}{P_t(s)}+ h^i_t)}, 
\end{cases}
\end{equation}
where  $\eta_t(s)$,  $\boldsymbol  \eta_t$, $\{\{h^i_t\}_{i=1}^{n(s)}\}_{s=1}^S$ and    $\mathbf h_t$ are scalars, and  $P_t(s)$  and $\bar{\mathbf P}_t$  are symmetric matrices with appropriate dimensions, $s \in \mathbb{N}_S$.   We  now establish a property  of the above  anastz.
\begin{Lemma}\label{lemma:p1}
The following relationship holds  for  any sub-population $s \in \mathbb{N}_S$ at any time $t \in [0,T]$,  
\begin{align}\label{eq:P1}
 \TR\big(\sum_{i=1}^{n(s)} \sum_{i'=1}^{n(s)} \nabla_{\Delta x^i_t \Delta x^{i'}_t} \psi_t(\mathbf y_t) \Sigma^{w}_t(s) \big( \frac{d}{dt} \Exp{(\Delta w^i_t)^\intercal \Delta w^{i'}_t} \big) \big)
 = \TR\big(\sum_{i=1}^{n(s)} \nabla_{\Delta x^i_t \Delta x^i_t} \psi_t(\mathbf y_t) \Sigma^{w}_t(s) \big).
\end{align}
\end{Lemma}
\begin{proof}
From~\eqref{eq:ansatz_value}, it results that  for  any  sub-population $s \in \mathbb{N}_S$ at  time $t \in [0,T]$, one has: 
\begin{align*}\label{eq:P1}
 &\TR\big(\sum_{i=1}^{n(s)} \sum_{i'=1}^{n(s)} \nabla_{\Delta x^i_t \Delta x^{i'}_t} \psi_t(\mathbf y_t) \Sigma^{w}_t(s) \big( \frac{d}{dt} \Exp{(\Delta w^i_t)^\intercal \Delta w^{i'}_t} \big) \big)= \TR\big(\sum_{i=1}^{n(s)} \sum_{i'=1}^{n(s)}   \frac{\nabla_{\Delta x^i_t} \psi_t(s)(\Delta x^i_t) (\nabla_{\Delta x^{i'}_t} \psi_t(s)(\Delta x^{i'}_t))^\intercal }{\psi_t(s)(\Delta x^i_t) \psi_t(s)(\Delta x^{i'}_t)}      \psi_t(\mathbf y_t) \\
 &\times  \Sigma^{w}_t(s) ( \frac{d}{dt} \Exp{(\Delta w^i_t)^\intercal \Delta w^{i'}_t} ) \big) =\TR\big( \psi_t(\mathbf y_t)  \sum_{i=1}^{n(s)} \sum_{i'=1}^{n(s)}    \frac{\nabla_{\Delta x^i_t} \psi_t(s)(\Delta x^i_t) (\nabla_{\Delta x^{i'}_t} \psi_t(s)(\Delta x^{i'}_t))^\intercal }{\psi_t(s)(\Delta x^i_t) \psi_t(s)(\Delta x^{i'}_t)}       \Sigma^{w}_t(s)  \\
& \times \big( \frac{d}{dt} \Exp{(\Delta w^i_t)^\intercal \Delta w^{i'}_t} \big)   \big) \substack{(a)\\=}\TR \big( \psi_t(\mathbf y_t) \sum_{i=1}^{n(s)} \sum_{i'=1}^{n(s)}  \frac{\nabla_{\Delta x^i_t} \psi_t(s)(\Delta x^i_t) (\nabla_{\Delta x^{i'}_t} \psi_t(s)(\Delta x^{i'}_t))^\intercal }{\psi_t(s)(\Delta x^i_t) \psi_t(s)(\Delta x^{i'}_t)} \\
& \times (-\sum_{j=1}^{f(s)} \frac{\alpha^{i,j}(s) \alpha^{i',j}(s)}{n(s)} \Sigma^{w}_t(s)   \big) + \TR\big(  \psi_t(\mathbf y_t) \sum_{i=1}^{n(s)}    \frac{\nabla_{\Delta x^i_t} \psi_t(s)(\Delta x^i_t) (\nabla_{\Delta x^{i}_t} \psi_t(s)(\Delta x^{i}_t))^\intercal }{\psi_t(s)(\Delta x^i_t) \psi_t(s)(\Delta x^{i}_t)}    \Sigma^{w}_t(s)  \big)\substack{(b)\\=}\TR\big(  \sum_{i=1}^{n(s)}    \\
& \frac{\nabla_{\Delta x^i_t} \psi_t(s)(\Delta x^i_t) (\nabla_{\Delta x^{i}_t} \psi_t(s)(\Delta x^{i}_t))^\intercal }{\psi_t(s)(\Delta x^i_t) \psi_t(s)(\Delta x^{i}_t)}  \psi_t(\mathbf y_t)   \Sigma^{w}_t(s)  \big)\substack{(c)\\=} \TR\big(\sum_{i=1}^{n(s)} \nabla_{\Delta x^i_t \Delta x^i_t} \psi_t(\mathbf y_t) \Sigma^{w}_t(s) \big), 
\end{align*}
where $(a)$  follows from Lemmas~\ref{lemma:linear_dep} and~\ref{lemma:correlation}; (b) follows from the fact that the first term of the left-hand side is zero  due to the linear dependency in Lemma~\ref{lemma:linear_dep} and $\frac{\nabla_{\Delta x^i_t} \psi_t(s)(\Delta x^i_t)}{\psi_t(s)(\Delta x^i_t)}=2\lambda P_t(s)\Delta x^i_t$, and  $(c)$ follows from~\eqref{eq:ansatz_value}. 
\end{proof}

 \subsection{Solution of Problem~\ref{prob1}}
For any time $t \in [0,T)$, define the following backward equations:
\begin{equation}\label{eq:Riccati}
\begin{cases}
- \dot P_t(s)=Q_t(s) + P_t(s) A_t(s) + {A_t(s)}^\intercal P_t(s)  -P_t(s)  \big(B_t(s) (R_t(s))^{-1} {B_t(s)}^\intercal 
- 2 \lambda \frac{\mu(s)}{n(s)} \Sigma^{w}_t(s) \big) P_t(s), \quad s \in \mathbb{N}_S,\\
- \dot  {\bar{\mathbf P}}_t=\bar{\mathbf Q}_t + \bar{\mathbf P}_t \bar{\mathbf A}_t + {\bar{\mathbf A}_t}^\intercal \bar{\mathbf P}_t 
  -\bar{\mathbf P}_t\big(\bar{\mathbf B}_t {\bar{\mathbf R}_t}^{-1} {\bar{\mathbf B}_t}^\intercal- 2 \lambda {\bar{\boldsymbol \Sigma}}^w_t \big) \bar{\mathbf P}_t, 
\end{cases}
\end{equation} 
with boundary conditions $P_T(s)=Q_T(s)$ and $\bar{\mathbf P}_T=\bar{\mathbf Q}_T$. 

In the sequel, we refer to~\eqref{eq:Riccati} as (risk-sensitive) \emph{deep}  Riccati equation (DRE),   which identifies a scalable solution for the large-scale optimization problems described in Subsection~\ref{sec:problem_statement}.  It is to be noted that the   dimensions of the matrices in DRE  do not depend on the number of agents $n(s)$, $\forall s \in \mathbb{N}_S$.  Interestingly,  DRE  in  team setting consists of $S$ local Riccati equations and one global Riccati equation whereas such decomposition  does not generally  hold for  the   game setting~\cite{Jalal2019Automatica}. For  the special case of weakly coupled agents in Definition~\ref{def:weakly}, DRE decomposes further into $S+\sum_{s=1}^{f(s)}$ smaller equations, where $\bar{\mathbf P}_t=\DIAG(\mu(s)\DIAG(\bar P^j_t(s))_{j=1}^{f(s)})_{s=1}^S$  such that  for every $j \in \mathbb{N}_{f(s)}$ and  $s \in \mathbb{N}_S$,
\begin{equation}\label{eq:Riccati_weakly}
\begin{cases}
- \dot{\bar P}^j_t(s)=Q_t(s) + \bar Q^j_t(s)+ \bar P^j_t(s) (A_t(s)+\bar A^j_t(s)) 
+ {(A_t(s)+\bar A^j_t(s))}^\intercal \bar P^j_t(s)  -\bar P^j_t(s)  \big((B_t(s)+\bar B^j_t(s))  (R_t(s)\\
\qquad +\bar R^j_t(s))^{-1} {(B_t(s)+\bar B^j_t(s))}^\intercal 
- 2 \lambda \frac{\mu(s)}{n(s)} \Sigma^{w}_t(s) \big) \bar P^j_t(s).
\end{cases}
\end{equation} 
 \begin{Remark}[Computational complexity]
\emph{The space complexity of DRE is independent of the number of agents $n(s)$, $s \in \mathbb{N}_S$,  and   is  quadratic  with respect to the number of features and sub-populations,  i.e. $\mathcal{O}((\sum_{s=1}^S f(s))^2)$. For the special case of  weakly coupled agents,  the complexity reduces to  linear order with respect to the number of features and sub-populations, i.e. $\mathcal{O}((\sum_{s=1}^S f(s))$. When $f(s) \ll n(s)$, the above complexity reductions are significant. Notice that  it is always  possible to consider a larger number of features and sub-populations in order  to model more complex  couplings; however,  this is a trade-off between  under-fitting  and  over-fitting  of a model  and between   cheap  and  expensive  (numerical)  implementation of an algorithm.}
\end{Remark}

For any agent $i \in \mathbb{N}_{n(s)}$,  sub-population $s \in \mathbb{N}_S$ and time $t \in [0,T)$, define the followings from the solution of DRE:
\begin{equation}\label{eq:gains}
\begin{cases}
 - \dot \xi^i_t:=\big( A_t(s) - B_t(s) {R_t(s)}^{-1} {B_t(s)}^\intercal P_t(s) \\
 +  2 \lambda \frac{\mu(s)}{n(s)} \Sigma^{w}_t(s) P_t(s) \big)^\intercal \xi^i_t 
 -Q_t(s) (r^i_t -\sum_{j=1}^{f(s)} \alpha^{i,j}(s) \bar r^j_t(s)),\\
  - \dot{\boldsymbol \xi}_t:=\big( \bar{\mathbf A}_t - \bar{\mathbf B}_t {\bar{\mathbf R}_t}^{-1} \bar{\mathbf B}_t^\intercal \bar{\mathbf P}_t +  2 \lambda {\bar{\boldsymbol \Sigma}}^w_t \bar{\mathbf  P}_t \big)^\intercal \boldsymbol \xi_t \\
 - \DIAG(\mu(1)\bar{\mathbf Q}_t(1),\ldots,\mu(S)\bar{\mathbf Q}_t(S))\bar{\mathbf r}_t,\\
  \theta^n_t(s):=- (R_t(s))^{-1} (B_t(s))^\intercal P_t(s),\\
\rho^{n,i}_t:=-(R_t(s))^{-1} (B_t(s))^\intercal  \xi_t^i,\\
\ROW(
\bar{\boldsymbol \theta}^n_t(1), \ldots,\bar{\boldsymbol \theta}^n_t(S)):=\bar{\boldsymbol \theta}^n_t=-\bar{\mathbf R}_t^{-1} \bar{ \mathbf B}_t^\intercal \bar{\mathbf P}_t,\\
\VEC(\bar{\boldsymbol \rho}^n_t(1), \ldots,\bar{\boldsymbol \rho}^n_t(S)):=\bar{{\boldsymbol \rho}}^n_t=-\bar{\mathbf R}_t^{-1} \bar{ \mathbf B}_t^\intercal \boldsymbol \xi_t,\\
\ROW(\bar{\boldsymbol \theta}^{n,1}_t(s),\ldots,\bar{\boldsymbol \theta}^{n,f(s)}_t(s)):=\bar{\boldsymbol \theta}^n_t(s),\\
\VEC(\bar{\boldsymbol \rho}^{n,1}_t(s),\ldots,\bar{\boldsymbol \rho}^{n,f(s)}_t(s)):=\bar{{\boldsymbol \rho}}^n_t(s),
\end{cases}
\end{equation}
where  $\xi^i_T=Q_T(s) (r^i_T -\sum_{j=1}^{f(s)} \alpha^{i,j}(s) \bar r^T_t(s))$ and  $\boldsymbol \xi_T:= \DIAG(\mu(1)\bar{\mathbf Q}_T(1),\ldots,\mu(S)\bar{\mathbf Q}_T(S))\bar{\mathbf r}_T$. 

\begin{Theorem}\label{thm:mfs}
Let Assumption~\ref{assumption1}   hold. The optimal solution of Problem~\ref{prob1} is described  as follows. For  any agent $i \in \mathbb{N}_{n(s)}$ in  sub-population $s \in \mathbb{N}_S$ at any time $t \in [0,T]$,
\begin{align}\label{eq:DSS_strategy}
u^{i,\ast}_t&= \theta^n_t(s) x^i_t - \sum_{j=1}^{f(s)} \alpha^{i,j}(s) \theta^n_t(s) \bar x^j_t(s)  +\rho^{n,i}_t + \sum_{j=1}^{f(s)} \alpha^{i,j}(s) \bar{\boldsymbol \theta}^{n,j}_t(s) \bar{\mathbf x}_t +\sum_{j=1}^{f(s)} \alpha^{i,j}(s) \bar{\boldsymbol \rho}^{n,j}_t(s),
\end{align}
where the above gains and corrections terms are  given by~\eqref{eq:gains}.
\end{Theorem}
\begin{proof}
From Lemma~\ref{lemma:correlation}, it is observed that for  any $i \in \mathbb{N}_{n(s)}$ $j \in \mathbb{N}_{f(s)}$,  $s \in \mathbb{N}_S$ and  $t \in [0,T]$,  $ \Exp{(\Delta w^i_t)^\intercal \bar w_t^j(s)}=  \Exp{ (\bar w_t^j(s))^\intercal \Delta w^i_t}=~0$. Therefore, the following holds:
\begin{equation}\label{eq:p2}
\Compress
 \TR\big(\sum_{i=1}^{n(s)} \sum_{j=1}^{f(s)}  \nabla_{\Delta x^i_t \bar x_t^j(s)} \psi_t(\mathbf y_t) \Sigma^{w}_t(s) \big(\frac{d}{dt} \Exp{(\Delta w^i_t)^\intercal \bar w_t^j(s)} \big)\big)=0.
\end{equation}
According to Lemma~\ref{lemma:p1} and equation~\eqref{eq:p2}, the HJB equation~\eqref{eq:HJB}  can be  decomposed into smaller equations as follows.  For any agent $i \in \mathbb{N}_{n(s)}$ in sub-population $s \in \mathbb{N}_S$, there is an HJB equation similar to   a linear exponential quadratic problem with the state $\Delta x^i_t$, action $\Delta u^i_t$, the dynamics characterized by matrices  $A_t(s)$ and $B_t(s)$,  a zero-mean noise with  the covariance matrix $ (\mu(s)/n(s)) \Sigma^{w}_t(s)$, and a  tracking signal $\Delta r^i_t$,  where the optimal action is given by:
 $\Delta u^{i,\ast}_t= -(R_t(s))^{-1} {B_t(s)}^\intercal P_t(s) \Delta x^i_t - (R_t(s))^{-1} (B_t(s))^\intercal \xi^i_t$,
 where $P_t(s)$ and $\xi^i_t$ are  presented in~\eqref{eq:Riccati} and~\eqref{eq:gains}, respectively.  In addition, the corresponding parameters of the ansatz~\eqref{eq:ansatz_value}  can be computed as: 
$
  - \dot \eta_t(s) = \lambda \frac{\mu(s)}{n(s)} \eta_t(s) \TR(P_t(s) \Sigma_t^{w}(s) )$ and 
$h^i_t=\int_{\tau=t}^T( \Norm{\Delta r^i_\tau}{Q_\tau(s)} -  \| \xi^i_\tau\|_{\tilde Q_\tau(s)}) d\tau$,
where  $\eta_T(s)=1$ and $\tilde Q_\tau(s):=B_\tau(s) (R_\tau(s))^{-1}
{B_\tau(s)}^\intercal 
- 2\lambda \frac{\mu(s)}{n(s)} \Sigma^{w}_\tau(s)$. Furthermore, there is a global  HJB equation similar to  a    linear exponential quadratic problem with  state $\bar{\mathbf x_t}$, action $\bar{\mathbf u}_t$, the  dynamics  characterized  by matrices $\bar{\mathbf A}_t$ and $\bar{\mathbf B}_t$,  a zero-mean noise with the covariance matrix $\bar{\boldsymbol  \Sigma}^w_t$, and a tracking signal $\bar {\mathbf r}_t$,  where the optimal action is expressed by:
$\bar{\mathbf u}^\ast_t= - \bar{\mathbf R}^{-1}_t \bar{\mathbf B}_t^\intercal \bar{\mathbf P}_t \bar{\mathbf x}_t - \bar{\mathbf R}_t^{-1} \bar{\mathbf B}_t^\intercal \boldsymbol \xi_t$,
where  $\bar{\mathbf P}_t$ and $\boldsymbol \xi^i_t$ are  given  by~\eqref{eq:Riccati} and~\eqref{eq:gains}, respectively. Also,   the corresponding parameters of the ansatz~\eqref{eq:ansatz_value}  can be described as: 
$
  - \dot {\boldsymbol \eta}_t = \lambda \boldsymbol \eta_t \TR(\bar{\mathbf P}_t {\bar{\boldsymbol \Sigma}}^w_t)$,
 and $\mathbf h_t=  \int_{\tau=t}^T (\Norm{\bar {\mathbf r}_\tau}{\DIAG(\mu(1)\bar{\mathbf Q}_\tau(1),\ldots,\mu(S)\bar{\mathbf Q}_\tau(S))}  -\Norm{\boldsymbol \xi_\tau}{\tilde{\mathbf Q}_\tau})  d\tau$,
where   $\boldsymbol \eta_T=1$ and  $\tilde{\mathbf Q}_\tau:=\bar{\mathbf B}_\tau {\bar{\mathbf R}_\tau}^{-1}
\bar{\mathbf B}_\tau^\intercal-2 \lambda \bar{\boldsymbol \Sigma}^w_\tau$.

Finally, we describe the solution in terms of  the original variables according to~\eqref{eq:gauge}, i.e., the optimal solution of  agent $i$ in sub-population $s$ at time $t$ can be expressed as: $u^{i,\ast}_t= \Delta u^{i,\ast}_t+ \sum_{j=1}^{f(s)} \alpha^{i,j}(s) \bar u^{j,\ast}_t(s)$. Since the optimal  centralized strategy is implementable under DSS, it  means that it  is also the optimal DSS strategy. 
  \end{proof}

Prior to the operation of the system,   any  agent  $i \in \mathbb{N}_{n(s)}$ in sub-population $s \in \mathbb{N}_S$  can independently  solve two Riccati equations: a local  one (i.e. $P_t(s)$) and a  global one (i.e. $\bar{\mathbf P}_t$) for any time $t \in [0,T]$. During the control process,  any agent $i$ in sub-population $s$ coordinates itself within its sub-population at any  time $t \in [0,T)$ based on several factors: (a) the solution of local Riccati equation $P_t(s)$; (b)  local (private) information $\{x^i_t, r^i_t,\{\alpha^{i,j}(s)\}_{j=1}^{f(s)}\}$, and (c) global (public) information $\{\bar{\mathbf x}_t(s), \bar{\mathbf  r}_t(s), \lambda,  \Sigma^{w}_t(s), \mu(s), n(s)\}$.  At the same time,  agent $i$  coordinates itself within  the population based on other factors: (d) the solution of global Riccati equation~$\bar{\mathbf P}_t$ and (e) public information  $\{\bar{\mathbf x}_t, \bar{\mathbf  r}_t,\{\mu(s),\Sigma^{w}_t(s),n(s)\}_{s=1}^S\}$.  It is to be noted  that the only piece of  information that  needs to be   shared  at any time $t$ is  the joint deep state $\bar{\mathbf x}_t$, whose size is independent of the number of agents in each sub-population. 

\begin{Corollary}
Suppose the agents are weakly coupled according to Definition~\ref{def:weakly}.
\begin{itemize}
\item There is no gain of optimality in  distributing   the deep states of other sub-populations among agents. In particular, at any time $t \in [0,T)$,  every agent    needs only to know  its local state and  deep states of its sub-population, where  $
u^{i,\ast}_t= \theta^n_t(s) x^i_t + \sum_{j=1}^{f(s)} \alpha^{i,j}(s) (\theta^n_t(s)-\bar{\boldsymbol{\theta}}^{n,j}_t(s)) \bar x^j_t(s)+ \rho^{n,i}_t+ \sum_{j=1}^{f(s)} \alpha^{i,j}(s) \bar{\boldsymbol \rho}^{n,j}_t(s)
$, $\forall i \in \mathbb{N}_{n(s)}, \forall s \in \mathbb{N}_S$.

\item For the special case of  risk-neural cost function, where $\lambda \rightarrow 0$,  the optimal solution~\eqref{eq:DSS_strategy} is independent of macroscopic influence factors $\mu(s)$, $ s \in \mathbb{N}_S$.

\item The second condition in Assumption~\ref{assumption1}.(b) simplifies to   $(B_t(s)+\bar{B}^j_t(s)) (\bar R_t(s)+\bar{ R}^j_t(s))^{-1} (B_t(s)+\bar B^j_t(s))^\intercal - 2\lambda \frac{\mu(s)}{n(s)} \Sigma^w_t(s)$, $j \in \mathbb{N}_{f(s)},s \in \mathbb{N}_S$, being positive definite. 

\end{itemize} 
\end{Corollary}

A  salient property of the risk-neutral case is the fact that certainty equivalence theorem holds.  In particular, when the risk factor  in Theorem~\ref{thm:mfs} is set to zero,  the solution of Theorem~\ref{thm:mfs}   reduces to the solution of  the  risk-neutral problem, where the optimal strategy is independent of  the  number of agents in each sub-population $n(s)$, $s \in \mathbb{N}_S$, as well as the probability distribution of  driving noises (e.g.,  correlated and non-Gaussian) which is not the case for the risk-sensitive optimal strategy in Theorem~\ref{thm:mfs}. To  further emphasize the  complexity  of the risk-sensitive case compared to  the risk-neutral one,   consider a case wherein the  local noises are correlated. In this case,  key properties  described in Lemma~\ref{lemma:p1} and equation~\eqref{eq:p2}  do not necessarily  hold,  and hence  the decomposition proposed in Theorem~\ref{thm:mfs}  will not  hold either.

\subsection{Risk-neutral case}  

  In this subsection, we  show that the risk-sensitive solution converges to  the risk-neutral one as the number of agents increases. In addition, we introduce  (extra) optimization factors for the risk-neutral cost function.

\subsubsection{Price of robustness}The following assumption is imposed to ensure that the state dynamics and cost  function remain bounded as the number of agents goes to infinity.
\begin{Assumption}\label{assumption3}
All matrices in the dynamics~\eqref{eq:dynamics1} and   cost functions~\eqref{eq:cost1},  including covariance matrices of initial states and local noises, are independent of $n(s)$, $\forall s \in \mathbb{N}_S$.
\end{Assumption}

\begin{Definition}[Price of Robustness]
The price of robustness (PoR) is defined  to quantify the loss of performance by taking the robustness into account, i.e.
$PoR(\lambda):=\tilde J^{n,\lambda}(\mathbf g^\ast) -\lim_{\lambda' \rightarrow 0} \tilde J^{n,\lambda'}(\mathbf g^\ast)$,
where superscripts $\lambda$ and $\lambda'$ refer to the dependency of~\eqref{eq:total_cost_risk} to the risk parameter.
\end{Definition}

\begin{Corollary}\label{cor:risk}
Let Assumptions~\ref{assumption1} and \ref{assumption3} hold.  From Theorem~\ref{thm:mfs} and equations~\eqref{eq:Riccati} and~\eqref{eq:gains}, it follows  that the price of robustness converges to zero as $n(s)\rightarrow \infty$, $\forall s \in \mathbb{N}_S$. 
\end{Corollary}

\subsubsection{Optimization factors} Let  $\beta^i(s) \in (0,\infty)$  be an optimization factor for every agent $i \in \mathbb{N}_{n(s)}$ in sub-population $s \in \mathbb{N}_S$, where the cost function~\eqref{eq:total_cost} is generalized to
\begin{equation}\label{eq:total_cost_extra}
\bar J^n_T:=\sum_{s=1}^S  \frac{\mu(s)}{n(s)} \sum_{i=1}^{n(s)} \beta^{i}(s) J^{i}_T.
\end{equation}
To simplify the presentation, without loss of generality, it is assumed that  influence and optimization factors are normalized as follows: $\frac{1}{n(s)}\sum_{i=1}^{n(s)} \frac{\alpha^{i,j}(s)}{\sqrt{\beta^i(s)}}  \frac{\alpha^{i,j'}(s)}{\sqrt{\beta^i(s)}}=\ID{j= j'}$, $j,j' \in \mathbb{N}_{f(s)}$, $s \in \mathbb{N}_S$. We modify the gauge transformation~\eqref{eq:gauge} such that
$
 \Delta x^i_t:=x^i_t - \sum_{j=1}^{f(s)} \frac{\alpha^{i,j}(s)}{\beta^i(s)} \bar x_t^j(s)$,   $\Delta u^i_t:=u^i_t - \sum_{j=1}^{f(s)} \frac{\alpha^{i,j}(s)}{\beta^i(s)} \bar u_t^j(s)$,
   $\Delta w^i_t:=w^i_t - \sum_{j=1}^{f(s)} \frac{\alpha^{i,j}(s)}{\beta^i(s)} \bar w_t^j(s)$, $
      \Delta r^i_t:=r^i_t - \sum_{j=1}^{f(s)} \frac{\alpha^{i,j}(s)}{\beta^i(s)} \bar r_t^j(s).
$

\begin{Assumption}\label{ass:decoupled}
The agents are dynamically decoupled, i.e., $\bar A^j_t(s)$ and $\bar B^j_t(s)$, $s\in \mathbb{N}_S$, $t \in [0,T]$,  are zero.
\end{Assumption}

\begin{Corollary}\label{cor:optimization}
Let Assumption~\ref{assumption1}.(a) and Assumption~\ref{ass:decoupled} hold. For  cost function~\eqref{eq:total_cost_extra}, the optimal solution is given by: 
\begin{align}\label{eq:DSS_strategy_extra}
u^{i,\ast}_t&= \theta_t(s) x^i_t - \sum_{j=1}^{f(s)} \frac{\alpha^{i,j}(s)}{\beta^i(s)} \theta_t(s) \bar x^j_t(s) +\rho^{n,i}_t+ \sum_{j=1}^{f(s)} \frac{\alpha^{i,j}(s)}{\beta^i(s)}  \bar{\boldsymbol \theta}^{j}_t(s) \bar{\mathbf x}_t  +\sum_{j=1}^{f(s)}  \frac{\alpha^{i,j}(s)}{\beta^i(s)}  \bar{\boldsymbol \rho}^{n,j}_t(s).
\end{align}
\end{Corollary}
\begin{proof} 
The proof is similar to that of Theorem~\ref{thm:mfs}, where the certainty equivalence theorem and  modified gauge transformation under Assumption~\ref{ass:decoupled} leads to the decomposition of the centralized HJB equation.  For more details, see~\cite{arabneydi2016new}[Chapter 3] that uses a similar argument.
\end{proof}

\begin{Remark}
\emph{The extension to optimization factors does not generally hold for coupled dynamics and/or risk-sensitive case.}
\end{Remark}

\subsection{Connection to deep neural networks}\label{sec:NN}

The mathematics of  deep neural networks  can be traced back to the seminal work of Gauss~\cite{gauss1809theoria} in regression theory and its  application to  the study of  the cat's visual cortex in physiology~\cite{hubel1959receptive,hubel1962receptive}. Despite the recent progresses in deep learning~\cite{lecun2015deep,schmidhuber2015deep},  there are still many fundamental challenges  that need to be addressed such as performance guarantee and interpretability, high number of parameters and tunability, troubleshooting,   prior knowledge and small data.

In what follows, we highlight some aspects of deep structured  teams that can be useful for deep learning, and more importantly, for deep reinforcement learning. For simplicity of  presentation,  consider a special case of one sub-population with dynamically decoupled agents that is discretized with a sampling time $\Delta t$ and zero-order hold such that  $t=k \Delta t$, where  variables are indexed by a non-negative integer $k$. Therefore, the dynamics~\eqref{eq:dynamics1} can be expressed as:
$x^i_{k+1}= (\mathbf I+A_k\Delta t) x^i_k+ (B_k\Delta t) u^i_k+ w^i_k$,
where $w^i_k \sim \mathcal{N}(0, \Sigma^w_k \Delta t)$. From Theorem~\ref{thm:mfs},  the  dynamics of  agent $i \in \mathbb{N}_n$ under the  optimal  strategy is described by:
\begin{equation}\label{eq:NN_structure}
 x^i_{k+1}=\sum_{m=1}^n W^{i,m}_k(\lambda, \boldsymbol \alpha) x^m_k+b^i_k(\lambda, \boldsymbol \alpha)+w^i_k,
\end{equation}
where the optimal weight matrix and  bias term are as follows: for every $i,m \in \mathbb{N}_n$, $i \neq m$, 
\begin{equation}
\begin{cases}
W^{i,i}_k(\lambda, \boldsymbol \alpha):=\mathbf I+A_k \Delta t + B_k \Delta t \big( (1- \frac{1}{n} \sum_{j=1}^f (\alpha^{i,j})^2)  \theta^n_k + \frac{1}{n} \sum_{j=1}^f \sum_{j'=1}^f \alpha^{i,j} \alpha^{i,j'} \bar{ \theta}^{n,j,j'}_k\big),\\
W^{i,m}_k(\lambda, \boldsymbol \alpha):=- B_k \Delta t \big(\frac{1}{n}\sum_{j=1}^f \alpha^{i,j}\alpha^{m,j}\theta^n_k + \frac{1}{n}\sum_{j=1}^f \sum_{j'=1}^f \alpha^{i,j}\alpha^{m,j'} \bar{\theta}^{n,j,j'}_k)\big),\\
 b^i_k(\lambda, \boldsymbol \alpha):=B_k \Delta t (\rho^{n,i}_t + \sum_{j=1}^f\alpha^{i,j}  \bar{\boldsymbol \rho}_k^{n,j}).
 \end{cases}
\end{equation}

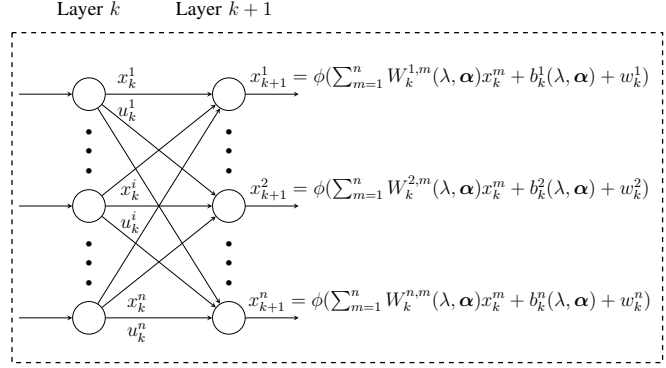
\begin{figure}
\hspace{2cm}
\scalebox{0.62}{
\begin{tikzpicture}[x=1.5cm, y=1.2cm, >=stealth]
  \draw[font=\large](0,3cm) node {Layer $k$} ;
\draw[font=\large](2.9cm,3cm) node {Layer $k+1$} ;
\draw[black, thick, dashed] (-1.1,2.1) rectangle (8.2,-3.8);

\foreach \m/\l [count=\y] in {1,missing,2,missing,3}
  \node [every neuron/.try, neuron \m/.try] (input-\m) at (0,2-\y) {};

\foreach \m/\l [count=\y] in {1,missing,2,missing,3}
  \node [every neuron/.try, neuron \m/.try ] (output-\m) at (2,2-\y) {};

  \draw [<-] (input-1) -- ++(-1,0)
    node [above,midway] {\hspace{3.5cm}\large $ x^1_k$}
        node [below,midway] {\hspace{3.5cm}\large $u^1_k$};
     \draw [<-] (input-2) -- ++(-1,0)
    node [above,midway] {\hspace{3.5cm} \large $x^i_k$}
        node [below,midway] {\hspace{3.5cm} \large $u^i_k$};
        
          \draw [<-] (input-3) -- ++(-1,0)
    node [above,midway] {\hspace{3.8cm} \large $x^{n}_k$}
        node [below,midway] {\hspace{3.8cm} \large $u^{n}_k$};

  \draw [->]  (output-1) -- ++(1,0)  node[above, midway]  
   {\hspace*{7.5cm} \large $x^1_{k+1}= \phi(\sum_{m=1}^n W^{1,m}_k(\lambda, \boldsymbol \alpha) x^m_k+ b^1_k(\lambda, \boldsymbol \alpha) +w^1_k)$};   
      \draw [->]  (output-2) -- ++(1,0)  node[above, midway]         {\hspace*{7.5cm} \large $x^2_{k+1}= \phi(\sum_{m=1}^n W^{2,m}_k(\lambda, \boldsymbol \alpha) x^m_k+ b^2_k(\lambda, \boldsymbol \alpha) +w^2_k)$}; 
    
      \draw [->]  (output-3) -- ++(1,0)  node[above, midway]         {\hspace*{7.5cm} \large $x^n_{k+1}= \phi(\sum_{m=1}^n W^{n,m}_k(\lambda, \boldsymbol \alpha) x^m_k+ b^n_k(\lambda, \boldsymbol \alpha) +w^n_k)$}; 
    \foreach \i in {1,...,3}
  \foreach \j in {1,...,3}
    \draw [->] (input-\i) -- (output-\j);    
\end{tikzpicture}
}
\caption{The optimal solution of a deep structured team resembles a deep feed-forward neural network, where  optimal gains and biases are obtained analytically while taking into account regularization, scalability, uncertainty and robustness of the design. Such an explicit representation can  be useful in formulating more complex deep neural networks.
}
\end{figure}

A salient property of the  structure~\eqref{eq:NN_structure} is  that it is interpretable, i.e.,  it  provides a one-to-one relationship between the parameters of a deep neural network and those of the control system  presented in~Section~\ref{sec:problem}.  In addition, since weight initialization is  a key step in training a deep neural network,  it is possible to use  the above optimal weight matrix for computing  initial conditions (using  a deep Riccati equation for a network  consisting of $n$ neurons with depth $T/ \Delta t$,  where each neuron  has its own influence factor  and nominal operating point).   Moreover,  to efficiently  incorporate a certain design, it is sometimes necessary to simultaneously adjust the weights at all layers  in a consistent manner. For example,  the robustness parameter $\lambda$ can automatically tune  the weights  in such a way that  the resultant  design  is  resilient to  uncertainties.  Also, using   large penalty functions on control actions can  promote sparsity.  Such  design specifications are  important in ensuring that the  network is not over-fitted (i.e., there is room for uncertainties) or it does not contain excessive   number of coefficients.  Furthermore, the above results can   be  extended to  \emph{constrained} deep structured teams,  where  quadratic programming and MPC techniques can be utilized to tackle the control problem.  For instance, rectified linear unit (ReLU) activation function can be viewed as a sufficient condition to impose an inequality  constraint (i.e. $ x^i_k \geq 0$, $i \in \mathbb{N}_n$)  such that:
$x^i_{k+1}=\phi(\sum_{m=1}^n W^{i,m}_k(\lambda, \boldsymbol \alpha) x^m_k+b^i_k(\lambda, \boldsymbol \alpha))$,
where $\phi(\boldsymbol \cdot ):=\max(0,\boldsymbol \cdot)$. This perspective is consistent with the fact that  ReLU activation function has been successfully used  in  classification problems, e.g., image processing,  where  the states of interest  are in the form of probability (frequency) of the occurrence of a feature (which is a non-negative number). Thus,  ReLU   function may be interpreted as a sufficient condition to ensure that the  feasible states are non-negative. 

\section{Main results with partial deep-state sharing}\label{sec:main_pdss}
In this section, we  propose two sub-optimal solutions  for PDSS structure under Assumption~\ref{assumption3} and the following standard assumption on the initial states.
\begin{Assumption}\label{ass:initial_state}
The initial states of each sub-population $s \notin \mathcal{O}$ are mutually independent with   identical mean such that $\Exp{x^i_0}=\Exp{x^{i'}_0}$, $\forall i,i' \in \mathbb{N}_{n(s)}$, $\forall s \notin \mathcal O$.
\end{Assumption}
To quantify the performance of each sub-optimal solution, we define the price of information.
\begin{Definition}[Price of Information]
The price of information (PoI) of a PDSS strategy~$\hat{\mathbf g}$ is defined as the performance gap between strategy $\hat{\mathbf g}$ and optimal DSS strategy $\mathbf g^\ast$, i.e.,
$\Delta \tilde J^n_T(\hat{\mathbf g}):=\tilde J_T^{n}(\hat{\mathbf g}) - \tilde J_T^{n}(\mathbf g^\ast)$.
\end{Definition}

To distinguish the proposed sub-optimal solutions from  the optimal one, let $\hat x^i_t$ and $\hat u^i_t$ denote the state and action of agent $i \in \mathbb{N}_{n(s)}$ in sub-population $s \in \mathbb{N}_S$ under the sub-optimal strategy $\hat{\mathbf{g}}$ at any time $t \in [0,T]$. For any feature $j \in \mathbb{N}_{f(s)}$ in sub-population $s \in \mathbb{N}_S $, define
$\hat {\bar x}_t^j(s):=\frac{1}{n(s)}\sum_{i=1}^{n(s)} \alpha^{i,j}(s) \hat x^i_t$ and $\hat {\bar u}_t^j(s):=\frac{1}{n(s)}\sum_{i=1}^{n(s)} \alpha^{i,j}(s) \hat u^i_t$.

Let $\mathbf z^n_t:=\VEC(\VEC(z_t^{n,1}(s),\ldots,z_t^{n,f(s)}(s)))_{s=1}^S$ be an  estimate of $\bar{\mathbf x}_t$ such that  $\mathbf z^n_1(s):=\bar{\mathbf  x}_1(s)$, $s \in \mathcal{O}$,  $\mathbf z^n_1(s):=\Exp{\bar{\mathbf  x}_1(s) }$, $s \notin \mathcal{O}$, and  for any $t \in (0,T]$,
\begin{equation}\label{eq:finite_estimate}
\begin{cases}
\mathbf z^n_t(s)=\bar{\mathbf x}_t(s), & s \in \mathcal{O},\\
\dot{\mathbf z}^n_t(s)=(\bar{\mathbf A}_t(s)+\bar{\mathbf B}_t(s)\bar{\boldsymbol \theta}^n_t)\mathbf z^n_t+ \bar{\mathbf B}_t(s) \bar{\boldsymbol \rho}^n_t, & s \notin \mathcal{O}.
\end{cases}
\end{equation}
Based on the finite-population  estimate in~\eqref{eq:finite_estimate},  we propose a  PDSS strategy such that the action of any agent $i \in \mathbb{N}_{n(s)}$  in  sub-population $s \in \mathbb{N}_S$ at any time $t \in [0,T]$ is given by:
\begin{equation}\label{eq:PDSS1_strategy}
\begin{cases}
\hat  u^i_t= \theta^n_t(s) x^i_t - \sum_{j=1}^{f(s)} \alpha^{i,j}(s) \theta^n_t(s)  z^{n,j}_t(s) +\rho^{n,i}_t  + \sum_{j=1}^{f(s)} \alpha^{i,j}(s) \bar{\boldsymbol \theta}^{n,j}_t(s) \mathbf z^n_t +\sum_{j=1}^{f(s)} \alpha^{i,j}(s) \bar{\boldsymbol \rho}^{n,j}_t(s),\\
\dot{\mathbf  z}^n_t(s)= (\bar{\mathbf A}_t(s)+\bar{\mathbf B}_t(s)\bar{\boldsymbol \theta}^n_t)\mathbf z^n_t+ \bar{\mathbf B}_t(s) \bar{\boldsymbol \rho}^n_t   +H(s)(\dot{\bar{\mathbf x}}_t(s) -(\bar{\mathbf A}_t(s)+\bar{\mathbf B}_t(s)\bar{\boldsymbol \theta}^n_t) \mathbf z^n_t- \bar{\mathbf B}_t(s) \bar{\boldsymbol \rho}^n_t  ),
\end{cases}
\end{equation}
where the update rule of the estimate in~\eqref{eq:finite_estimate} is expressed in a Kalman-like fashion such that  the observer gain $H(s)=0$ if $s \notin \mathcal{O}$  and $H(s)=1$ if $s \in \mathcal{O}$.   Although~this estimator resembles a Kalman filter,  it is fundamentally  different from the standard Kalman filter and it is not necessarily the best possible estimate.  On the other hand,  it provides a measurable estimate  $\mathbf z^n_t$  with respect to a non-classical  PDSS information structure.   The interested reader is referred to~\cite[Theorem 7]{Jalal2019TSNE} for more details on the above Kalman-like filter  that emerges in the trade-off between data collection and data estimation. 

Alternatively,  it is  possible to  ignore the number of agents as well as the covariance matrices of sub-populations $s \notin \mathcal{O}$ and propose a Kalman-like filter similar to~\eqref{eq:PDSS1_strategy} based on the  infinite-population model such that the  action of agent $i \in \mathbb{N}_{n(s)}$, $s \in \mathbb{N}_S$, at time $t \in [0,T]$ is described by:
\begin{equation}\label{eq:PDSS2_strategy}
\begin{cases}
\hat  u^i_t= \theta^\infty_t(s) x^i_t - \sum_{j=1}^{f(s)} \alpha^{i,j}(s) \theta^\infty_t(s)  z^{\infty,j}_t(s) +\rho^{\infty,i}_t  + \sum_{j=1}^{f(s)} \alpha^{i,j}(s) \bar{\boldsymbol \theta}^{\infty,j}_t(s) \mathbf z^{\infty}_t +\sum_{j=1}^{f(s)} \alpha^{i,j}(s) \bar{\boldsymbol \rho}^{\infty,j}_t(s),\\
\dot{\mathbf  z}^\infty_t(s)= (\bar{\mathbf A}_t(s)+\bar{\mathbf B}_t(s)\bar{\boldsymbol \theta}^\infty_t)\mathbf z^\infty_t+ \bar{\mathbf B}_t(s) \bar{\boldsymbol \rho}^\infty_t  +H(s)(\dot{\bar{\mathbf x}}_t(s) -(\bar{\mathbf A}_t(s)+\bar{\mathbf B}_t(s)\bar{\boldsymbol \theta}^\infty_t) \mathbf z^\infty_t- \bar{\mathbf B}_t(s) \bar{\boldsymbol \rho}^\infty_t  ),
\end{cases}
\end{equation}
where the  above matrices, drifts and estimates are computed based on an infinite-population model, where $n(s)$ is set to infinity  for every sub-population $s \notin \mathcal{O}$.  Let  $\hat{\mathbf g}^n$ and  $\hat{\mathbf g}^\infty$ denote the PDSS  strategies~\eqref{eq:PDSS1_strategy} and \eqref{eq:PDSS2_strategy}, respectively.

\begin{Theorem}\label{thm:pmfs}
Let Assumptions~\ref{assumption1},~\ref{assumption3} and~\ref{ass:initial_state} hold. For Problem~\ref{prob2}, PDSS  strategies~\eqref{eq:PDSS1_strategy}  and~\eqref{eq:PDSS2_strategy}  provide  two different  solutions, where their  corresponding prices of information converge to zero as $n^\ast$ goes to  infinity, i.e.,
$\lim_{n^\ast \rightarrow \infty} \Delta J^n_T(\hat{\mathbf g}^n)=0$ and $\lim_{n^\ast \rightarrow \infty} \Delta J^n_T(\hat{\mathbf g}^\infty)=0$.
\end{Theorem}
\begin{proof}

Define a relative error between the deep states of sub-population $s \in \mathbb{N}_S$ at time $t \in [0,T]$ under the optimal strategy~\eqref{eq:DSS_strategy}, i.e. $ \bar{\mathbf  x}_t(s)$, and their estimates under  the sub-optimal strategy~\eqref{eq:PDSS1_strategy}, i.e. $\mathbf z_t^{n}(s)$,  such that  
\begin{equation}\label{eq:error_1}
e_t(s):= \bar{\mathbf x}_t(s)-\mathbf z_t^n(s).
\end{equation}
 Similarly, define  a relative  error between the deep states of sub-population $s \in \mathbb{N}_S$ at time $t$ under  under the sub-optimal strategy~\eqref{eq:PDSS1_strategy}, i.e. $\hat{\bar{\mathbf{x}}}_t(s)$, and  the estimate $\mathbf z_t^n(s)$, i.e.,
 \begin{equation}\label{eq:error_2} 
 \xi_t(s):=\hat{\bar{\mathbf{x}}}_t(s)-\mathbf z_t^n(s).
 \end{equation}
By definition, $e_0(s)=\xi_0(s)=\mathbf{0}$, $s \in \mathcal{O}$. Let $\mathbf e_t=\VEC(e_t(1),\ldots,e_t(S))$ and $\boldsymbol \xi_t=\VEC(\boldsymbol \xi_t(1),\ldots,\xi_t(S))$. From Theorem~\ref{thm:mfs}, the dynamics of deep states in sub-population $s \in \mathbb{N}_S$ under the optimal strategy is given by:
\begin{equation}\label{eq:optimal_dynamics_proof}
\Compress
d(\bar{\mathbf x}_t(s))=(\bar{\mathbf A}_t(s)+\bar{\mathbf B}_t(s)\bar{\boldsymbol \theta}^n_t)\bar{\mathbf x}_tdt+ \bar{\mathbf B}_t(s) \bar{\boldsymbol \rho}^n_t(s)+ \bar{\mathbf C}_t(s)d(\bar{\mathbf w}_t(s)).
\end{equation} 
Hence, from equations~\eqref{eq:mf1},~\eqref{eq:dynamics1},~\eqref{eq:augmented_variables},~\eqref{eq:DSS_strategy},~\eqref{eq:PDSS1_strategy},~\eqref{eq:error_1},~\eqref{eq:error_2} and~\eqref{eq:optimal_dynamics_proof}, and after some algebraic manipulations,  it follows that for any sub-population $s \notin \mathcal{O}$ at time $t \in (0,T)$,
\begin{equation}\label{eq:dynamics_error_1}
\begin{cases}
d(e_t(s))= (\bar{\mathbf A}_t(s)+\bar{\mathbf B}_t(s)\bar{\boldsymbol \theta}^n_t) \mathbf e_t dt+ \bar{\mathbf C}_t(s)d(\bar{\mathbf w}_t(s)),\\
d(\xi_t(s))=(\bar{\mathbf A}_t(s)+\bar{\mathbf B}_t(s)\boldsymbol \theta_t^n) \boldsymbol \xi_t dt+ \bar{\mathbf C}_t(s)d(\bar{\mathbf w}_t(s)),
\end{cases}
\end{equation}
where $\boldsymbol \theta_t^n:=\DIAG(\DIAG(\theta^n_t(s))_{f(s)})_{s=1}^S$
 and for any  $s \in \mathcal{O}$, 
 \begin{equation}\label{eq:dynamics_error_2}
\begin{cases}
\dot e_t(s)= (\bar{\mathbf A}_t(s)+\bar{\mathbf B}_t(s)\bar{\boldsymbol \theta}^n_t) \mathbf e_t -(\bar{\mathbf A}_t(s)+\bar{\mathbf B}_t(s)\boldsymbol \theta_t^n) \boldsymbol \xi_t ,\\
\xi_t(s)=\Zero_{f(s)d^s_x \times f(s)d^s_x}.
\end{cases}
\end{equation} 

It  can be concluded that the  three sample paths $\bar{\mathbf x}_t$, $\hat{\bar{\mathbf x}}_t$ and $\mathbf z^n_t$, $t \in [0,T]$, converge to the same path with probability one, as $n^\ast \rightarrow \infty$.  This is due the fact that (a)  the relative errors~\eqref{eq:error_1} and~\eqref{eq:error_2} have linear dynamics according to~\eqref{eq:dynamics_error_1} and~\eqref{eq:dynamics_error_2}; (b) matrices $\bar{\mathbf A}_t(s)$ and $\bar{\mathbf B}_t(s)$ are independent of $n(s)$, $s \notin \mathcal{O}$, from Assumption~\ref{assumption3}; (c) initial states and local noises are mutually independent with covariances matrices independent of $n(s)$;  (d)  $\lim_{n^\ast \rightarrow \infty} \theta^n_t(s)$ and  $\lim_{n^\ast \rightarrow \infty} \bar{\boldsymbol \theta}^n_t(s)$ exist because~\eqref{eq:Riccati} and~\eqref{eq:gains} are uniformly bounded and continuous in $n(s)$, and  (e)  from the strong law of large numbers, the following limits exist with   probability one for every $s \notin \mathcal{O}$: $\lim_{n^\ast \rightarrow \infty} \bar{\mathbf x}_0(s)=\lim_{n^\ast \rightarrow \infty} \hat{\bar{\mathbf x}}_0(s)=\lim_{n^\ast \rightarrow \infty} \mathbf z^n_0(s)$ and $\lim_{n^\ast \rightarrow \infty} \bar{\mathbf w}_t(s)=\mathbf 0$,  where $\lim_{n^\ast \rightarrow \infty}e_0(s)=\lim_{n^\ast \rightarrow \infty}\xi_0(s)=\mathbf 0$. Subsequently,    the two sample paths $\{\{x^i_t,u^i_t\}_{i=1}^{n(s)}\}_{s=1}^S$, governed by the DSS strategy~\eqref{eq:DSS_strategy}, and $\{\{\hat x^i_t,\hat u^i_t\}_{i=1}^{n(s)}\}_{s=1}^S$, by the PDSS strategy~\eqref{eq:PDSS1_strategy}, will  converge to the same path  with probability one, as $n^\ast \rightarrow \infty$. On the other hand,  the  cost function~\eqref{eq:cost1} is a continuous function in states and actions, and matrices $\bar{\mathbf Q}_t$ and $\bar{\mathbf R}_t$ are independent of $n(s)$, $s \in \mathbb{N}_S$.  Consequently,  the price of information of strategy~\eqref{eq:PDSS1_strategy}  converges to zero, as  $n^\ast \rightarrow \infty$.

The convergence proof for strategy~\eqref{eq:PDSS2_strategy} follows a similar argument, where  from the triangle inequality, one has:
\begin{equation}\label{eq:proof_2_aux}
\Delta \tilde J^n_T(\hat{\mathbf g}^\infty) \leq \Delta \tilde J^n_T(\hat{\mathbf g}^n)+ |\tilde J^n_T(\hat{\mathbf g}^\infty) -\tilde J^n_T(\hat{\mathbf g}^n) |.
\end{equation}
The first term of the right-hand side of~\eqref{eq:proof_2_aux} converges to zero, as shown above, and the second term of the right-hand side of~\eqref{eq:proof_2_aux} converges to zero because $\tilde J^n_T$ is  continuous and uniformly bounded with respect to the number of agents, and strategy $\hat{\mathbf g}^n$ converges to  $\hat{\mathbf g}^\infty$ with probability one,   as  $n^\ast \rightarrow \infty$,  because $\lim_{n^\ast \rightarrow \infty} \mathbf z_t^n=\mathbf z_t^\infty$,  $\lim_{n^\ast \rightarrow \infty} \boldsymbol{\theta}^n_t=\boldsymbol{\theta}^\infty_t$ and $\lim_{n^\ast \rightarrow \infty} \bar{\boldsymbol{\theta}}^n_t=\bar{\boldsymbol{\theta}}^\infty_t$.
\end{proof}

\begin{Remark}
\emph{Although both finite- and infinite-model PDSS strategies~\eqref{eq:PDSS1_strategy} and~\eqref{eq:PDSS2_strategy} converge to the same unique solution as $n^\star \rightarrow \infty$, they have  subtle differences. For example,  strategy \eqref{eq:PDSS1_strategy} takes the number of agents and covariance matrices  into account while the strategy~\eqref{eq:PDSS2_strategy}  ignores such information, which potentially leads to smaller price of information. On the other hand, strategy~\eqref{eq:PDSS2_strategy} is simpler for analysis because the influence factor of an individual  agent is asymptotically vanishing   in the infinite-population model. 
}
\end{Remark}

\section{Infinite-horizon cost function}\label{sec:generalizations}
In this section, we extend the results of  Theorems~\ref{thm:mfs} and~\ref{thm:pmfs}   to the infinite-horizon case.  Suppose  that the model is  time-homogeneous, and that the  cost function is given by:
\begin{equation}\label{eq:cost_infinite_horizon}
\tilde J^n_\infty(\mathbf g):=\limsup_{T \rightarrow \infty } \frac{1}{T} \tilde J^n_T(\mathbf g).
\end{equation}
The following   stability assumption is imposed on the model.
\begin{Assumption}\label{ass:infinite}
Let $(A(s),B(s))$ and $(\bar{\mathbf A},\bar{\mathbf B})$ be stablizable,  and $(A(s),(Q(s))^{1/2})$ and $(\bar{\mathbf A}, \bar{\mathbf Q}^{1/2})$ be detectable, $\forall s \in \mathbb{N}_S$. In addition,  there  are  positive definite matrices  $P(s)$, $\forall s \in \mathbb{N}_S$, and $\bar{\mathbf P}$ that solve the algebraic counterpart of the generalized Riccati equations~\eqref{eq:Riccati}. 
\end{Assumption}

\begin{Theorem}~\label{cor:infinite}
Let Assumptions~\ref{assumption1} and~\ref{ass:infinite}  hold.  
\begin{itemize}
\item Theorem~\ref{thm:mfs} extends to the infinite-horizon cost function~\eqref{eq:cost_infinite_horizon}  wherein the  strategy~\eqref{eq:DSS_strategy}  becomes stationary.  

\item Let also Assumptions~\ref{assumption3} and~\ref{ass:initial_state} hold.   Theorem~\ref{thm:pmfs}  extends to  the infinite-horizon cost function~\eqref{eq:cost_infinite_horizon} under an additional condition that matrix
 $\bar{\mathbf  A} +\bar{\mathbf B} \DIAG(\DIAG(\theta^n(s))_{f(s)})_{s=1}^S$ is Hurwitz. For the special case presented in Assumption~\ref{ass:decoupled}, the additional condition is automatically satisfied. 
\end{itemize}
\end{Theorem}
\begin{proof}
 The proof of the first part  follows directly from the standard stability and detectability conditions in~\cite{bacsar2008h},
  where  the algebraic forms of the Riccati equations in~\eqref{eq:Riccati} admit positive bounded solutions.  The proof of the second part, however, follows from an additional  condition to guarantee that the relative errors in~\eqref{eq:error_1} and~\eqref{eq:error_2} will remain bounded as  $T \rightarrow \infty$.  If   $\bar{\mathbf  A} +\bar{\mathbf B} \DIAG(\DIAG(\theta^n(s))_{f(s)})_{s=1}^S$ is Hurwitz,  the dynamics of the errors in~\eqref{eq:dynamics_error_1} and~\eqref{eq:dynamics_error_2} become stable, implying that their limits  exist. Moreover, when the dynamics of agents  are decoupled (i.e., $\mathbf  A=\DIAG(\DIAG(A(s))_{f(s)})_{s=1}^S$ and $\mathbf B=\DIAG(\DIAG(B(s))_{f(s)})_{s=1}^S$),   $\bar{\mathbf  A} +\bar{\mathbf B} \DIAG(\DIAG(\theta^n(s))_{f(s)})_{s=1}^S$  becomes  Hurwitz due to the fact that $A(s)+B(s) \theta^n_t(s)$ is  Hurwitz  for any  $s \in \mathbb{N}_S$  (which is a stabilizing solution of  the Riccati equations~\eqref{eq:Riccati} under Assumption~\ref{ass:infinite}).
\end{proof}

\section{A Supply Chain Management Example }\label{sec:numerical}

Consider a supplier  that provides a particular  product to  its consumers  (e.g., the  bandwidth provided by a telecommunication company). The product is distributed to  the consumers through a number of  distributors (hubs), each of which has its own operating  capacity.  The objective is to find  a risk-sensitive solution for the supplier  and distributors such that  the delivered  product    is as close as possible  to the supplier's  production level while the  distributors' demands  are fulfilled.

To this end, let the supplier be the only agent in the first sub-population, i.e. $n(1)=1$. Denote by  $x_t^1 $ and  $u_t^1$,   respectively,  the  production level and control output of the supplier at time $t \in [0,T]$,  normalized with respect to the number of distributors. Let  $w_t^1$  be the uncertainty  of  the market   at time~$t$.  The  state evolution of the supplier in~\eqref{eq:dynamics1}  is described by: $A(1)=0.4$, $B(1)=0.8$ and $C(1)=0.6$.
The second sub-population is comprised of $n(2)$  distributors, where $x^i_t$, $u^i_t$ and $w^i_t$ denote the state,  action and  uncertainty of  distributor $i \in \mathbb{N}_{n(2)}$ at time $t \in [0,T]$, respectively.  The   dynamics of each distributor in~\eqref{eq:dynamics1} is  expressed by $A(2)=2$, $B(2)=1$ and $C(2)=1$.  Let $r^i$ denote  the desired  operating point of distributor $i$,  which is uniformly chosen from the interval  $[0,1]$,  $\forall i \in \mathbb{N}_{n(2)}$.  In addition, denote by $\alpha^i$  the influence factor  of distributor $i \in \mathbb{N}_{n(2)}$, indicating its  contribution in delivering the product such that
$
\bar x_t(2)=\frac{1}{n(2)}\sum_{i=1}^{n(2)}\alpha^i x^i_t,
$
where $n(2) \bar x_t(2)$ is the total distributed (delivered)  products to consumers.  We add  a penalty  function to  the supplier's cost function for the mismatch between the production level $n(2) x^1_t$ and distributed  products $n(2) \bar x_t(2)$ such that the cost function of the supplier is
$
 J_T^{1}:=   (x^1_T)^2 + 0.5n(2)(x^1_T-\bar x_T(2))^2  
+   \int_{0}^T \big( (x^1_t)^2 + 0.5 n(2)(x^1_t-\bar x_t(2))^2    + (u^1_t)^2 \big)dt
$.  The cost function of  distributor $i \in \mathbb{N}_{n(2)}$ is  
$
J^{i}_T=   (x^i_T-r^i)^2   +   \int_{0}^T ((x^i_t-r^i)^2+  0.1(u^i_t)^2)dt$.
 From Theorem~\ref{thm:mfs},  DSS strategy~\eqref{eq:DSS_strategy} minimizes the  cost function~\eqref{eq:total_cost_risk}, given risk factor $\lambda=1$,   $\mu(1)=\frac{n(1)}{n(1)+n(2)}$, $\mu(2)=\frac{n(2)}{n(1)+n(2)}$.

  Since   the complexity of the proposed strategy in~\eqref{eq:DSS_strategy} is independent of the number of distributors,  we choose a relatively small $n(2)$ in our  simulations  for ease  of display.  We use  $n(2)=20$, terminal time $T=10$, and   sampling time  $0.01$.     Figure~\ref{fig3}(a) demonstrates the  case where the distribution is primarily  performed  by one distributor.  In Figure~\ref{fig3}(b), on the other hand,  it is assumed that  two distributors  are more actively involved in the distribution function.  Figure~\ref{fig3}(c) demonstrates the case where half of the distributors are equally  more active  in the process, and Figure~\ref{fig3}(d) the case where the distribution is  carried out by all distributors homogeneously.

\begin{figure}[t!]
\centering
\hspace{0cm}
\includegraphics[trim={0 6cm 0 6cm},clip, width=\linewidth]{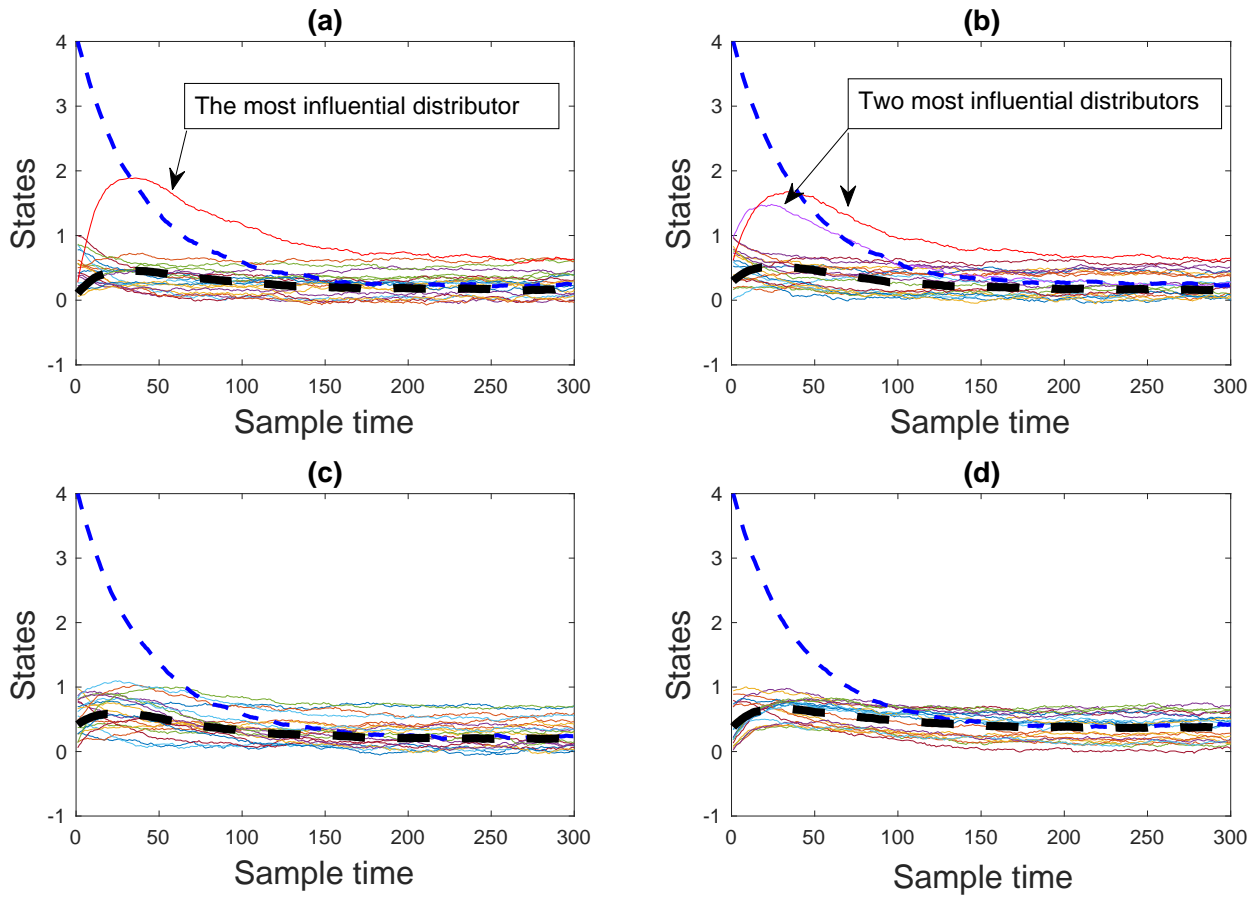}
\caption{The  trajectories of the states of  the supplier and distributors are depicted, where the dashed-blue  curve  is  the trajectory of the normalized  production level  of the supplier $x_t^1=\bar x_t(1)$,  the dashed-black  curve is  the normalized distributed product $\bar x_t(2)$,  and the remaining  curves  are the trajectories of  the distributors.  (a) The influence factor of one distributor is $4.45$ and all others  $0.1$;  (b)  influence factors of two distributors are  $3.14$ and all others are $0.1$;  (c)  half of the distributors  have the influence factor $0.1$ and the other half  $1.41$, and (d) all distributors have  the same influence factor.}\label{fig3} 
\end{figure}

\section{Conclusions}\label{sec:conclusion}

Inspired by deep learning  in big data analysis that  provides a data-driven model for  complex systems, we introduced  the notion of  deep structured  linear quadratic control problem in this paper.   It was shown how this notion can be used in the control of complex systems with a large number of decision makers, using  a low-dimensional deep Riccati equation. In particular, two non-classical information structures were studied, namely, deep-state sharing and partial deep-state sharing, where  the optimal solution for the former information structure and two sub-optimal solutions for the latter one  were obtained. The main results were also extended to infinite-horizon case.   In addition, the potential impact of the obtained results in enhancing our understanding of   deep  neural networks  was  demonstrated.  To illustrate the efficacy of the proposed results,  a  supply-chain  management  example   was provided. An interesting future work  is  to use reinforcement learning methods   to learn how to control such complex networked systems for the case when the model is not completely known.

\bibliographystyle{IEEEtran}
\bibliography{Jalal_Ref}
\appendices

\section{Proof of Proposition~\ref{prop:polynomial}}\label{sec:proof_prop:polynomial}
Since every  polynomial function of $F$ is invariant to right and left multiplication of matrix $F$,  it  results   that $F_x A_t=A_t F_x$ and $F_u B_t=B_t F_u$,  satisfying the equivariant dynamics condition. Similarly, $\TR(F_x^\intercal F_x Q_t x_tx_t^\intercal  )=\TR(x_t^\intercal F_x^\intercal Q_t F_x x_t)$ and  $\TR(F_x^\intercal F_x Q_t x_tx_t^\intercal  )=\TR(x_t^\intercal F_x^\intercal Q_t F_x x_t)$,  meeting the equivariant cost condition.

\section{Proof of Proposition~\ref{thm:normal}}\label{sec:proof_thm:normal}
From the spectral theorem, matrices $F$ and $F^\intercal$ can be expressed in terms of their eigenvalues and eigenvectors  such that $F=\sum_{j=1}^n \lambda_j v_j v_j^{\ast}$ and $F^\intercal =\sum_{j=1}^n  \lambda_j^\ast v_j v_j^{\ast}$, where $\langle v_j | v_j \rangle= v_j^\ast v_j=1$ and  $\langle v_j | v_{j'}\rangle=0$, $\forall j \neq j'$. Then, 
\begin{align}
\Norm{F x_t}{Q_t}&=\TR( x_t^\intercal F^\intercal  Q_t F x_t )=\TR( Q_t F x_t   x_t^\intercal F^\intercal )\\
&=\TR( Q_t (\sum_{j=1}^n  \lambda_j v_j v_j^{\ast} x_t)  (\sum_{j=1}^n  \lambda^\ast_j x_t^\intercal v_j v_j^{\ast}))\\
&=
 \TR( Q_t (\sum_{j=1}^n  \lambda_j v_j v_j^{\ast} x_t  v_j^\ast v_j)  (\sum_{j=1}^n  \lambda^\ast_j    x_t^\intercal v_j v_j^{\ast}))\\
&=\TR( Q_t (\sum_{j=1}^n  \langle \lambda_j| \lambda_j \rangle  \Proj(x_t,v_j) \Proj(x_t,v_j)^\ast))\\
&=\sum_{j=1}^n \langle \lambda_j| \lambda_j \rangle \Norm{\Proj(x_t,v_j)}{Q_t}.
\end{align}
Similar relationship holds for  the actions, i.e., $\Norm{F u_t}{R_t}=\sum_{j=1}^n \langle \lambda_j| \lambda_j \rangle (\Norm{\Proj(x_t,v_j)}{R_t}$.

\section{Proof of Proposition~\ref{proposition:symmetric}}\label{sec:proof_proposition:symmetric}
Since $F$ is a symmetric real-valued matrix, there exists a spectral decomposition $F=Z \Lambda Z^\intercal$, where  the orthogonal matrix $Z$ and diagonal matrix $\Lambda$ contain the eigenvectors and eigenvalues of $F$, respectively.
From Proposition~\ref{prop:polynomial},  one set of solutions  is polynomial functions of $F$, where  for any set of coefficients $ a_t(h) \in \mathbb{R}$, $ h \in \mathbb{N}_H \cup \{0\}$, $H \in \mathbb{N}$:
$A_t=\sum_{h=0}^{H}  a_t(h) F^h=\sum_{h=0}^{H} a_t(h) Z \Lambda^h Z^\intercal=\sum_{h=0}^{H} a_t(h) (\sum_{j=1}^n \lambda^h_j v_j v_j^\intercal)=a_t \mathbf I_{n \times n} +\sum_{j=1}^n \bar a^j_t v_j v_j^\intercal$,
where  $a_t:=a_t(0)$ and  $\bar a^j_t:=\sum_{h=1}^{H} \lambda_j^h a_t(h) $, $j \in \mathbb{N}_n$.   A similar argument holds for $B_t$, $Q_t$ and $R_t$.  The proof is completed on noting  that $v_j^\intercal x_t=\sqrt{n}\bar x_t^j$ and   $v_j^\intercal u_t=\sqrt{n} \bar u_t^j$.

\section{Proof of Proposition~\ref{proposition:permutation}}\label{sec:proof_proposition:permutation}
 The conditions of equivariant dynamics and cost function in Definition~\ref{Def:Invariant}  reduce to those of  exchangeable dynamics and cost in \cite[Definition 2.1]{arabneydi2016new} for  every permutation  matrix.   It is shown in~\cite[Proposition 2.1]{arabneydi2016new} that any exchangeable LQ model can be expressed as an  LQ model wherein  the agents are coupled through the mean (unweighed average) of the states and actions. To find a set of polynomial families that are equivariant to any  permutation matrix $F$, one must have $A_t F=F A_t$, $B_t F=F B_t$, $Q_t=F^\intercal Q_t F$ and $R_t=F^\intercal R_t F$. Therefore,  matrices $A_t$, $B_t$, $Q_t$ and $R_t$ must be in the following  form:  $c_t \mathbf{I}_{n \times n} + \bar c_t  \mathbf{1}_{n \times n}$, $c_t, \bar c_t \in \mathbb{R}$. Note  that  $F^\intercal F=\mathbf{I}_{n \times n}$ and  $\langle \lambda_j|\lambda_j \rangle=1$, $j \in \mathbb{N}_n$  for every permutation matrix $F$, implying that all features are equally important. 
\end{document}